\long\def\symbolfootnote[#1]#2{\begingroup\def\thefootnote{\fnsymbol{footnote}}\footnote[#1]{#2}\endgroup}

\documentclass{amsart}
\usepackage{amssymb}
\usepackage{amsmath}
\usepackage{amsfonts}

\newtheorem{theorem}{Theorem}[section]
\newtheorem{corollary}[theorem]{Corollary}
\newtheorem{lemma}[theorem]{Lemma}
\theoremstyle{remark}

\theoremstyle{definition}

\newtheorem{example}[theorem]{Example}
\theoremstyle{proposition}
\newtheorem{proposition}[theorem]{Proposition}
\numberwithin{equation}{section}

\begin{document}
\author{Ovidiu Munteanu and Jiaping Wang}
\title[Metric measure spaces]{Smooth metric measure spaces with nonnegative
curvature}
\date{}
\maketitle

\begin{abstract}
In this paper we study both function theoretic and spectral properties on
complete noncompact smooth metric measure space $(M,g,e^{-f}dv)$ with
nonnegative Bakry-\'{E}mery Ricci curvature. Among other things, we derive a
gradient estimate for positive $f$-harmonic functions and obtain as a
consequence the strong Liouville property under the optimal sublinear growth
assumption on $f.$ We also establish a sharp upper bound of the bottom
spectrum of the $f$-Laplacian in terms of the linear growth rate of $f.$
Moreover, we show that if equality holds and $M$ is not connected at
infinity, then $M$ must be a cylinder. As an application, we conclude steady
Ricci solitons must be connected at infinity.
\end{abstract}

\section{Introduction}

\symbolfootnote[0]{The first author has been partially supported by NSF grant No. DMS-1005484}

On a Riemannian manifold $\left( M,g\right),$ the consideration of weighted
measure of the form $e^{-f}dv,$ where $f$ is a smooth function and $dv$ is
the volume element induced by the metric $g,$ arises naturally in various
situations. It can be viewed as the volume form of a suitable conformal
change of the metric $g.$ Perhaps a more notable example is in the work of
Perelman \cite{P}, where he introduces a functional involving the integral
of the scalar curvature with respect to such a weighted measure and
formulates the Ricci flow as the gradient flow of the functional. The triple 
$(M,g, e^{-f}dv)$ is customarily called a smooth metric measure space. The
differential operator $\Delta_f,$ which is called $f$-Laplacian and given by 
\begin{equation*}
\Delta _{f}:=\Delta -\nabla f\cdot \nabla
\end{equation*}%
is more naturally associated with such a smooth metric measure space than
the classical Laplacian as it is symmetric with respect to the measure $%
e^{-f}dv.$ That is, 
\begin{equation*}
\int_{M}\left\langle \nabla \varphi ,\nabla \psi \right\rangle
e^{-f}=-\int_{M}\left( \Delta _{f}\varphi \right) \psi e^{-f},
\end{equation*}%
for any $\varphi ,\psi \in C_{0}^{\infty }\left( M\right) .$

Again, we point out that the operator $f$-Laplacian is very much related to
the Laplacian of a suitable conformal change of the background Riemannian
metric. It also appears as the generator of a class of stochastic diffusion
processes, the Brownian motion with drifts.

The Bakry-\'{E}mery Ricci tensor \cite{BE} of the metric measure space $%
(M,g,e^{-f}dv)$ is defined by 
\begin{equation*}
Ric_{f}:=Ric+Hess(f),
\end{equation*}%
where $Ric$ denotes the Ricci curvature of $M$ and $Hess(f)$ the Hessian of $%
f.$ This curvature relates to $f$-Laplacian via the following Bochner
formula 
\begin{equation*}
\Delta _{f}|\nabla u|^{2}=2|Hess(u)|^{2}+2\langle \nabla u,\nabla \Delta
_{f}u\rangle +2Ric_{f}(\nabla u,\nabla u).
\end{equation*}%
This of course suggests the important role of $Ric_{f}$ in the analysis of $%
f $-Laplacian. Perhaps as a more prominent example, $Ric_{f}$ also appears
in the study of the Ricci flow. The gradient solitons of the Ricci flow,
which arise from the singularity analysis of the Ricci flow, are defined to
be complete manifolds $(M,g)$ that the following equation 
\begin{equation*}
Ric_{f}=\lambda g
\end{equation*}%
holds for some function $f$ and constant $\lambda .$ Obviously, the Einstein
manifolds are gradient Ricci solitons. The gradient Ricci solitons are
called shrinking, steady and expanding accordingly when $\lambda >0,$ $%
\lambda =0$ and $\lambda <0,$ \cite{H}.

The classification of gradient Ricci solitons is an important problem from
the point of view of both the Ricci flow singularity analysis and purely as
a class of geometric partial differential equations. The problem has
received much attention recently. The book \cite{CLN} is a good source for
some of the important results. But it seems fair to say that the whole
picture is far from clear for now.

Partially motivated by interest in the study of gradient Ricci solitons,
various attempts have been made recently to study the geometry and analysis
on general metric measure spaces. We will refer the readers to \cite{Li, W}
for some of the results. It should be noted however a while back,
Lichnerowicz \cite{Lic} has already done some pioneering work in this
direction. In particular, he has extended the classical Cheeger-Gromoll
splitting theorem to the metric measure spaces with $Ric_{f}\geq 0$ and $f$
bounded.

In this paper, we will investigate some function theoretic and spectral
properties of metric measure space $\left(M,g,e^{-f}dv\right).$ While the
results are of independent interest, applications to the steady gradient
Ricci solitons are no doubt of our main focus. Throughout, we will assume $%
Ric_f\ge 0$ and $f$ is of linear growth unless otherwise noted. Recall that $%
f$ is of linear growth if for all $x$ in $M,$ 
\begin{equation*}
|f|(x)\le \alpha \,r(x)+\beta
\end{equation*}%
for some constants $\alpha$ and $\beta,$ where $r(x):=d(p,x)$ is the geodesic
distance function to a fixed point $p$ in $M.$ The linear growth rate $a$ of $f$
is then defined as the infimum of all such values $\alpha.$

Clearly, a steady gradient Ricci soliton $(M,g)$ satisfies $Ric_f=0$ for
some $f.$ It is also well-known that the potential function $f$ in this case
is of linear growth. So all of our results are applicable to the steady
gradient solitons. On the other hand, it should also be noted that many, if
not all, of our results will fail without the growth assumption on function $%
f.$

Our first result gives a gradient estimate for positive $f$-harmonic
functions on $(M,g,e^{-f}dv).$

\begin{theorem}
\label{A} Let $\left( M,g,e^{-f}dv\right) $ be a complete noncompact smooth
metric measure space with $Ric_{f}\geq 0.$ Assume that $f$ has linear growth
rate $a$ and let $u>0$ be $f$-harmonic on $M,$ i.e., $\Delta _{f}u=0$ on $M.$
Then the following gradient estimate holds true on $M.$%
\begin{equation*}
\left\vert \nabla \log u\right\vert \leq C\left( n\right) a,
\end{equation*}%
where the constant $C\left( n\right) $ depends only on the dimension $n.$ In
particular, if $f$ is of sublinear growth, then any positive $f$-harmonic
function $u$ on $M$ must be constant.
\end{theorem}

This estimate is sharp as demonstrated by the following simple example.

\begin{example}
\cite{N} \label{E1}Let $M=\mathbb{R}\times \mathbb{S}^{n-1}$ and $f\left(
t,\theta \right) =at$ for $t\in \mathbb{R}$ and $\theta \in \mathbb{S}^{n-1}$%
. Then $u\left( t,\theta \right) :=e^{at}$ is positive $f$-harmonic on $M.$
Clearly, the linear growth rate of $f$ is $a$ and $|\nabla \log u|=a.$
\end{example}

Although the statement in the theorem takes the form of Yau's classical
result on the positive harmonic functions, we would like to point out that
our proof is quite different. In the classical case of the Ricci curvature,
Yau \cite{Y} directly works with $\log u$ and obtains an estimate on $%
|\nabla \log u|$ via the Bochner formula. This approach works also with the
N-Bakry-\'{E}mery Ricci curvature given by

\begin{equation*}
Ric_{f}^{N}:=Ric+Hess_{f}-\frac{1}{N}df\otimes df
\end{equation*}%
as demonstrated by Li \cite{Li}. For the case of curvature $Ric_f,$ if one
imposes suitable assumption on $|\nabla f|$, say, it is bounded, then it is
still possible to utilize Yau's argument as shown by Wu \cite{Wu}. However,
with only the growth assumption on $f,$ this direct approach seems to run
into essential obstacles. Our argument relies on both Yau's idea and the
well known De Giorgi-Nash-Moser theory. In a recent paper \cite{B},
Brighton, by applying Yau's idea to function $u^{\epsilon}$ instead of $\log
u,$ proved bounded $f$-harmonic functions must be constant without any
assumption on $f$ so long as $Ric_f\ge 0.$ In our proof, we first refined
Brighton's argument and derived the gradient estimate under the assumption
that $u$ is of exponential growth. Here, no growth assumption on $f$ is
necessary. In particular, this implies that any sub-exponential growth
positive $f$-harmonic on $M$ with $Ric_f\ge 0$ is constant. The growth
assumption on $f$ was then used to get the desired growth control on $u.$
For that, we utilize a different set of techniques including a mean value
inequality obtained through the Moser's iteration argument.

We also deal with the polynomial growth $f$-harmonic functions. Here, the
result is very much parallel to the case of harmonic functions on a manifold
with nonnegative Ricci curvature, obtained by Cheng and Yau \cite{CY},
Li-Tam \cite{LT} , Li \cite{L1}, and Colding-Minicozzi \cite{CM},
respectively. We define the space%
\begin{equation*}
\mathcal{H}^{d}\left( M\right) :=\left\{ u:\ \Delta _{f}u=0\ \ \text{and\ \ }%
\left\vert u\right\vert \left( x\right) \leq C\left( r\left( x\right)
+1\right) ^{d}\right\} .
\end{equation*}

\begin{theorem}
\label{Dimension} Let $\left( M,g,e^{-f}dv\right) $ be a complete noncompact
smooth metric measure space with $Ric_{f}\geq 0$ and $f$ bounded. Then there
exists $\mu >0$ such that 
\begin{equation*}
\dim \mathcal{H}^{d}\left( M\right) =1,\ \ \ \ \text{if \ }d<1
\end{equation*}%
\begin{equation*}
\dim \mathcal{H}^{d}\left( M\right) \leq n+1,\ \ \ \text{if\ \ }\ d=1
\end{equation*}%
and 
\begin{equation*}
\dim \mathcal{H}^{d}\left( M\right) \leq Cd^{\mu }\ ,\ \ \ \text{if\ \ }\
d\geq 1.
\end{equation*}
\end{theorem}

Our second objective is to study the spectrum of the $f$-Laplacian on $%
(M,g,e^{-f}dv).$ Define $\lambda _{1}\left( M\right) :=\min Spec\left(
-\Delta_{f}\right)$ to be the bottom spectrum of $\Delta_f.$ By the
variational characterization, we have 
\begin{equation*}
\lambda _{1}\left( M\right) =\inf_{\phi \in C_{0}^{\infty }\left( M\right) }%
\frac{\int_{M}\left\vert \nabla \phi \right\vert ^{2}e^{-f}}{\int_{M}\phi
^{2}e^{-f}}.
\end{equation*}

The following result summarizes what we proved concerning the bottom
spectrum.

\begin{theorem}
\label{Estimate} Let $\left( M,g,e^{-f}dv\right) $ be a complete noncompact
smooth metric measure space with $Ric_{f}\geq 0.$ Then 
\begin{equation*}
\lambda _{1}\left( M\right) \leq \frac{1}{4}a^{2},
\end{equation*}%
where $a$ is the linear growth rate of $f$. Moreover, if $\lambda _{1}\left(
M\right) =\frac{1}{4}a^{2},$ then $M$ is connected at infinity; or $M$ is
isometric to $\mathbb{R}\times N$ for some compact manifold $N.$
\end{theorem}

Note that the splitting case in the theorem does occur. Indeed, for $M=%
\mathbb{R}\times N,$ if we take $f(t,y)=at$ for $(t,y)\in \mathbb{R}\times
N, $ then $\left\vert \nabla f\right\vert =a,$ $\Delta _{f}e^{\frac{1}{2}%
at}=-\frac{1}{4}a^{2}e^{\frac{1}{2}at}$ and $\lambda _{1}\left( M\right) =%
\frac{1}{4}a^{2}.$

Applying the preceding result to the gradient steady Ricci solitons, we
obtain the following.

\begin{corollary}
A nontrivial gradient steady Ricci soliton must be connected at infinity.
\end{corollary}

This is because for a steady gradient Ricci soliton $\left( M,g,f\right) $
such that $Ric_{f}=0$ on $M$ and $\sup |\nabla f|=a,$ one can show $\lambda
_{1}\left( M\right) =\frac{1}{4}a^{2}.$ Note now that the splitting case can
not arise as otherwise $Hess(f)=0$ and $M$ would be Ricci flat.

Historically, Cheng \cite{Cg} proved a sharp upper bound of the bottom
spectrum of the Laplacian on a complete manifold with Ricci curvature
bounded below. Later on, in \cite{LW2} and \cite{LW}, P. Li and the second
author studied the rigidity issue when the sharp upper bound is achieved.
Our results here are very much in the same spirit. Our arguments, however,
follow \cite{LW1} more closely.

The paper is organized as follows. In Section \ref{Vol_Comp} we discuss
Laplacian and volume comparison estimates and establish the upper bound
estimate for the bottom spectrum $\lambda _{1}\left(M\right).$ In Section %
\ref{harmonic} we prove the gradient estimate for positive $f$-harmonic
functions and related Liouville type results concerning $f$-harmonic
functions of polynomial growth. Finally, in Section \ref{Rigid} we discuss
the structure of manifolds with maximal bottom spectrum and the resulting
application to the steady Ricci solitons.

In a sequel to this paper, we will address some similar issues on smooth
metric measure spaces with $Ric_f$ bounded below.

The second author would like to thank Ben Chow for his interest in this work.

\section{Volume comparison theorems\label{Vol_Comp}}

In this section we discuss Laplacian and volume comparison estimates for
smooth metric measure spaces with nonnegative Bakry-\'{E}mery Ricci
curvature. The estimates in this section are instrumental in proving other
results of this paper. Also, as an immediate application, we obtain upper
bound estimates for the bottom spectrum of $\Delta _{f}.$

Let $\left( M,g,e^{-f}dv\right) $ be a smooth metric measure space. Take any
point $x\in M$ and denote the volume form in geodesic coordinates centered
at $x$ with 
\begin{equation*}
dV\left( \exp _{x}\left( r\xi \right) \right) =J\left( x,r,\xi \right)
drd\xi ,
\end{equation*}%
where $r>0$ and $\xi \in S_{x}M,$ a unit tangent vector at $x.$ It is well
known that if $y\in M$ is any point such that $y=\exp _{x}$ $\left( r\xi
\right),$ then 
\begin{equation*}
\Delta d\left( x,y\right) =\frac{J^{\prime }\left( x,r,\xi \right) }{J\left(
x,r,\xi \right) }\text{ \ and \ }\Delta _{f}d\left( x,y\right) =\frac{%
J_{f}^{\prime }\left( x,r,\xi \right) }{J_{f}\left( x,r,\xi \right) },
\end{equation*}%
where $J_{f}\left( x,r,\xi \right) :=e^{-f}J\left( x,r,\xi \right) $ is the
f-volume form in geodesic coordinates. For a fixed point $p\in M$ and $R>0,$
define 
\begin{equation}
A\left( R\right) :=\sup_{x\in B_{p}\left( 3R\right) }\left\vert f\right\vert
\left(x\right).  \label{v1}
\end{equation}%
For a set $\Omega,$ we will denote by $V\left( \Omega \right) $ the volume
of $\Omega $ with respect to the usual volume form $dv,$ and $V_{f}\left(
\Omega \right) $ the f-volume of $\Omega $. The following result has been
established in \cite{Ya}.

\begin{lemma}
\label{Vol_Ric_Pos} Let $\left( M,g,e^{-f}dv\right) $ be a smooth metric
measure space with $Ric_{f}\geq 0.$ Then along any minimizing geodesic
starting from $x\in B_{p}\left(R\right)$ we have 
\begin{equation*}
\frac{J_{f}\left( x,r_{2},\xi \right) }{J_{f}\left( x,r_{1},\xi \right) }%
\leq e^{4A}\left( \frac{r_{2}}{r_{1}}\right) ^{n-1}\text{ and \ }\frac{%
J\left( x,r_{2},\xi \right) }{J\left( x,r_{1},\xi \right) }\leq e^{6A}\left( 
\frac{r_{2}}{r_{1}}\right) ^{n-1}
\end{equation*}%
for any $0<r_{1}<r_{2}<R.$ In particular, for any $0<r_{1}<r_{2}<R,$ 
\begin{equation*}
\frac{V_{f}\left( B_{x}\left( r_{2}\right) \right) }{V_{f}\left( B_{x}\left(
r_{1}\right) \right) }\leq e^{4A}\left( \frac{r_{2}}{r_{1}}\right) ^{n}\text{
and \ }\frac{V\left( B_{x}\left( r_{2}\right) \right) }{V\left( B_{x}\left(
r_{1}\right) \right) }\leq e^{6A}\left( \frac{r_{2}}{r_{1}}\right) ^{n}.
\end{equation*}
Here, $A=A(R)$ as defined by (\ref{v1}).
\end{lemma}

\begin{proof}[Proof of Lemma \protect\ref{Vol_Ric_Pos}]
We include a proof here for the reader's convenience. Some of the
ingredients of this proof will also be used later.

Let $\gamma $ be the minimizing geodesic from $x$ to $y$ such that $\gamma
\left( 0\right) =x$ and $\gamma \left( r\right) =y.$ Recall the following
Laplace comparison theorem \cite{FLZ, W}, 
\begin{equation}
\Delta _{f}d\left( x,y\right) \leq \frac{n-1}{r}-\frac{2}{r^{2}}%
\int_{0}^{r}tf^{\prime }\left( t\right) dt,  \label{v0}
\end{equation}%
where $f\left( t\right) :=f\left( \gamma \left( t\right) \right) $.
Integrating by parts, we get%
\begin{equation}
\Delta _{f}d\left( x,y\right) \leq \frac{n-1}{r}-\frac{2}{r}f\left( r\right)
+\frac{2}{r^{2}}\int_{0}^{r}f\left( t\right) dt.  \label{v2}
\end{equation}
For $0<r_{1}<r_{2}<R$, integrating (\ref{v2}) from $r_1$ to $r_2$ yields 
\begin{gather}
\log \left( \frac{J_{f}\left( x,r_{2},\xi \right) }{J_{f}\left( x,r_{1},\xi
\right) }\right) \leq \left( n-1\right) \log \left( \frac{r_{2}}{r_{1}}%
\right)  \label{v3} \\
+2\int_{r_{1}}^{r_{2}}\frac{1}{r^{2}}\left( \int_{0}^{r}f\left( t\right)
dt\right) dr-2\int_{r_{1}}^{r_{2}}\frac{1}{r}f\left( r\right) dr.  \notag
\end{gather}%
However, 
\begin{equation*}
\int_{r_{1}}^{r_{2}}\frac{1}{r^{2}}\left( \int_{0}^{r}f\left( t\right)
dt\right) dr=-\frac{1}{r}\left( \int_{0}^{r}f\left( t\right) dt\right)
|_{r_{1}}^{r_{2}}+\int_{r_{1}}^{r_{2}}\frac{1}{r}f\left( r\right) dr.
\end{equation*}%
Plugging into (\ref{v3}), we conclude%
\begin{equation}
\frac{J_{f}\left( x,r_{2},\xi \right) }{J_{f}\left( x,r_{1},\xi \right) }%
\leq \left( \frac{r_{2}}{r_{1}}\right) ^{n-1} \exp\left(\frac{2}{r_{1}}%
\int_{0}^{r_{1}}f(t)dt- \frac{2}{r_{2}}\int_{0}^{r_{2}}f(t) dt\right).
\label{v4}
\end{equation}
Therefore, 
\begin{equation*}
\frac{J_{f}\left( x,r_{2},\xi \right) }{J_{f}\left( x,r_{1},\xi \right) }%
\leq e^{4A}\left( \frac{r_{2}}{r_{1}}\right) ^{n-1}
\end{equation*}%
for all $x\in B_{p}\left( R\right) $ and $0<r_{1}<r_{2}<R.$ Clearly, the
corresponding result for $J\left( x,r,\xi \right) $ follows by using again
the definition (\ref{v1}).

To establish the volume comparison, we use that 
\begin{equation*}
\frac{J_{f}\left( x,t,\xi \right) }{J_{f}\left( x,s,\xi \right) }\leq
e^{4A}\left( \frac{t}{s}\right) ^{n-1}
\end{equation*}%
for any $0<s<r_{1}<t<r_{2}<R.$ Integrating in $t$ from $r_{1}$ to $r_{2}$
and $s$ from $0$ to $r_{1},$ we get 
\begin{equation*}
\frac{V_{f}\left( B_{x}\left( r_{2}\right) \right) -V_{f}\left( B_{x}\left(
r_{1}\right) \right) }{V_{f}\left( B_{x}\left( r_{1}\right) \right) }\leq
e^{4A}\frac{\left( r_{2}\right) ^{n}-\left( r_{1}\right) ^{n}}{\left(
r_{1}\right) ^{n}}.
\end{equation*}%
This implies that 
\begin{equation*}
\frac{V_{f}\left( B_{x}\left( r_{2}\right) \right) }{V_{f}\left( B_{x}\left(
r_{1}\right) \right) }\leq e^{4A}\left( \frac{r_{2}}{r_{1}}\right)
^{n}+1-e^{4A}\leq e^{4A}\left( \frac{r_{2}}{r_{1}}\right) ^{n}.
\end{equation*}%
Now the volume comparison for $V\left( B_{x}\left( r\right) \right) $
follows directly from here. The Lemma is proved.
\end{proof}

Using Lemma \ref{Vol_Ric_Pos} we obtain a sharp upper bound for $\lambda
_{1}\left( M\right),$ the bottom spectrum of $\Delta _{f},$ by assuming $f$
is of linear growth. Recall 
\begin{equation*}
\lambda _{1}\left( M\right) :=\inf_{\phi \in C_{0}^{\infty }\left( M\right) }%
\frac{\int_{M}\left\vert \nabla \phi \right\vert ^{2}e^{-f}}{\int_{M}\phi
^{2}e^{-f}}.
\end{equation*}

\begin{theorem}
\label{Est_Ric_Pos} Let $\left( M,g,e^{-f}dv\right) $ be a complete
noncompact smooth metric measure space with $Ric_{f}\geq 0.$ If there exist
positive constants $a,b>0$ such that 
\begin{equation*}
\left\vert f\right\vert \left( x\right) \leq ar\left( x\right) +b\ \ \text{%
for all }x\in M,
\end{equation*}%
then we have the upper bound estimate 
\begin{equation*}
\lambda _{1}\left( M\right) \leq \frac{1}{4}a^{2}.
\end{equation*}%
In particular, if $f$ has sublinear growth, then $\lambda _{1}\left(
M\right) =0.$
\end{theorem}

\begin{proof}[Proof of Theorem \protect\ref{Est_Ric_Pos}]
By setting $x=p,$ $r_{1}=1$ and $r_{2}=R>1$ and noting $\left\vert
f\right\vert ( x) \leq ar\left( x\right) +b$ on $M,$ (\ref{v4}) implies

\begin{equation*}
J_{f}\left( p,R,\xi \right) \leq CR^{n-1}e^{-\frac{2}{R}\int_{0}^{R}f\left(
t\right) dt}\leq CR^{n-1}e^{aR}
\end{equation*}%
for all $R>1.$ Therefore, 
\begin{equation}
V_{f}\left( B_{p}\left( R\right) \right) \leq CR^{n}e^{aR}.  \label{v5}
\end{equation}%
We now claim $\lambda _{1}\left( M\right) \leq \frac{1}{4}a^{2}.$ Indeed,
take a cut-off $\psi $ on $B_{p}\left( R\right) $ such that $\psi =1$ on $%
B_{p}\left( R-1\right) ,$ $\psi =0$ on $M\backslash B_{p}\left( R\right) $
and $\left\vert \nabla \psi \right\vert \leq c.$ Consider $\phi \left(
y\right) :=e^{-\frac{1}{2}\left( a+\varepsilon \right) r\left( y\right)
}\psi \left( y\right) $ as a test function in the variational principle for $%
\lambda _{1}\left( M\right) $. Using the volume growth (\ref{v5}) we get
immediately that $\lambda _{1}\left( M\right) \leq \frac{1}{4}\left(
a+\varepsilon \right) ^{2}.$ Since $\varepsilon >0$ is arbitrary this
implies the desired estimate. The Theorem is proved.
\end{proof}

As indicated in the introduction, the estimate in Theorem \ref{Est_Ric_Pos}
is sharp. We now show that the bottom spectrum of steady Ricci solitons also
attains this upper bound. Recall a gradient steady Ricci soliton is a
manifold $\left( M,g,f\right) $ satisfying $R_{ij}+f_{ij}=0.$ It is known
that there exists a positive constant $a>0$ such that $\left\vert \nabla
f\right\vert ^{2}+S=a^{2},$ where $S$ is the scalar curvature of $M.$ It is
also known that $S\geq 0$ for any gradient steady Ricci soliton \cite{Ca,Ch}%
. Another useful relation is $\Delta f+S=0,$ which is obtained directly from 
$Ric_{f}=0$ by taking trace. In summary, a steady Ricci soliton satisfies
the following.%
\begin{eqnarray}
\left\vert \nabla f\right\vert ^{2}+S &=&a^{2},\text{\ \ for some constant }%
a>0  \label{v6} \\
\Delta f+S &=&0  \notag \\
S &\geq &0.  \notag
\end{eqnarray}

We first recall a well known result.

\begin{lemma}
\label{L1}Let $\left( M,g,e^{-f}dv\right) $ be a smooth metric measure
space. If there exists a positive function $v>0$ such that $\Delta _{f}v\leq
-\lambda v$ for some constant $\lambda >0,$ then $\lambda _{1}\left(
M\right) \geq \lambda .$
\end{lemma}

\begin{proof}[Proof of Lemma \protect\ref{L1}]
For completeness, we include a proof of this result. For an exhaustion of $M$
by compact domains $\Omega _{i}\subset \subset M,$ consider the first
Dirichlet eigenfunction $u_{i}.$ 
\begin{eqnarray*}
\Delta _{f}u_{i} &=&-\lambda _{1}\left( \Omega _{i}\right) u_{i}\ \ \ \text{%
in\ \ }\Omega _{i} \\
u_{i} &=&0\ \ \ \ \text{on }\partial \Omega _{i}.
\end{eqnarray*}%
It is known that we may assume $u_{i}>0.$ By the strong maximum principle, $%
\frac{\partial u_{i}}{\partial \eta }<0$ on $\partial \Omega _{i}.$ Now, 
\begin{gather*}
\left( \lambda _{1}\left( \Omega _{i}\right) -\lambda \right) \int_{\Omega
_{i}}vu_{i}e^{-f}\ge \int_{\Omega _{i}}\left( u_{i}\Delta _{f}v-v\Delta
_{f}u_{i}\right) e^{-f} \\
=\int_{\partial \Omega _{i}}\left( u_{i}\frac{\partial v}{\partial \eta }-v%
\frac{\partial u_{i}}{\partial \eta }\right) e^{-f}=-\int_{\partial \Omega
_{i}}v\frac{\partial u_{i}}{\partial \eta }e^{-f}>0.
\end{gather*}%
Since both $u_{i}$ and $v$ are positive, this shows $\lambda _{1}\left(
\Omega _{i}\right) \geq \lambda .$ The Lemma follows from 
\begin{equation*}
\lim_{i\rightarrow \infty }\lambda _{1}\left( \Omega _{i}\right) =\lambda
_{1}\left( M\right).
\end{equation*}
\end{proof}

\begin{proposition}
\label{Spec_Steady} Let $\left( M,g,f\right) $ be a gradient steady Ricci
soliton, normalized as in (\ref{v6}). Then $\lambda _{1}\left( M\right) =%
\frac{a^{2}}{4}.$
\end{proposition}

\begin{proof}[Proof of Proposition \protect\ref{Spec_Steady} ]
Since $S\geq 0,$ it follows that $\left\vert \nabla f\right\vert \leq a.$
Therefore, $\left\vert f \right\vert \left( x\right) \leq ar\left( x\right)
+b $ for any $x\in M.$ So by Theorem \ref{Est_Ric_Pos}, $\lambda _{1}\left(
M\right) \leq \frac{1}{4}a^{2}.$ To show the equality, we proceed as
follows. Observe that 
\begin{equation*}
\Delta _{f}e^{\frac{1}{2}f}=\left( \frac{1}{2}\Delta _{f}\left( f\right) +%
\frac{1}{4}\left\vert \nabla f\right\vert ^{2}\right) e^{\frac{1}{2}f}.
\end{equation*}%
But for a steady soliton, 
\begin{equation*}
\Delta _{f}\left( f\right) =\Delta f-\left\vert \nabla f\right\vert
^{2}=-\left( S+\left\vert \nabla f\right\vert ^{2}\right) =-a^{2}.
\end{equation*}%
Since $\left\vert \nabla f\right\vert \leq a,$ we conclude that 
\begin{gather*}
\Delta _{f}e^{\frac{1}{2}f}=\left( -\frac{1}{2}a^{2}+\frac{1}{4}\left\vert
\nabla f\right\vert ^{2}\right) e^{\frac{1}{2}f} \\
\leq \left( -\frac{1}{2}a^{2}+\frac{1}{4}a^{2}\right) e^{\frac{1}{2}f}=-%
\frac{a^{2}}{4}e^{\frac{1}{2}f}.
\end{gather*}%
By Lemma \ref{L1}, we have $\lambda _{1}\left( M\right) \geq \frac{a^{2}}{4}%
. $ This proves Proposition \ref{Spec_Steady}.
\end{proof}

\section{\ $f$-harmonic functions\label{harmonic}}

In this section we continue to assume $Ric_{f}\geq 0$ on $\left(
M,g,e^{-f}dv\right).$ The main objective is to derive the following global
gradient estimate for positive $f$-harmonic functions defined on $M.$ This
in particular leads to a strong Liouville theorem under optimal growth
assumption on $f.$

For a fixed point $p\in M,$ we let $r\left( x\right) :=d\left( p,x\right) .$

\begin{theorem}
\label{Grad_Est} Let $\left( M,g,e^{-f}dv\right) $ be a complete noncompact
smooth metric measure space with $Ric_{f}\geq 0$. Assume there exist
positive constants $a$ and $b$ such that 
\begin{equation*}
\left\vert f \right\vert \left( x\right) \leq a\, r\left( x\right) +b\ \text{%
on} \ M.
\end{equation*}%
Let $u>0$ be $f$-harmonic on $M.$ Then the following gradient estimate holds
true on $M.$%
\begin{equation*}
\left\vert \nabla \log u\right\vert \leq C\left( n\right) a,
\end{equation*}%
where constant $C\left( n\right) $ depending only on $n,$ the dimension of $%
M.$
\end{theorem}

An immediate consequence is the following strong Liouville property. As
noted by the explicit example in first section, the growth assumption on $f$
is optimal.

\begin{corollary}
Let $\left( M,g,e^{-f}dv\right) $ be a complete noncompact smooth metric
measure space with $Ric_{f}\geq 0.$ If $f$ is of sublinear growth, then any
positive $f$-harmonic function $u$ on $M$ is constant.
\end{corollary}

We prove Theorem \ref{Grad_Est} in several steps. First, we show that $%
\left\vert \nabla \log u\right\vert $ can be controlled from above by $\frac{%
1}{R}\sup_{B_{p}\left( R\right) }u$. This follows by adapting Yau's \cite{Y}
argument. We then verify that $u$ must be of exponential growth with the
exponent controlled by the constant $a.$ To achieve this, we use the Moser
iteration technique, following the ideas in \cite{G,SC}.

\begin{proposition}
\label{P} Let $\left( M,g,e^{-f}dv\right) $ be a complete noncompact smooth
metric measure space with $Ric_{f}\geq 0.$ Assume that there exist constants 
$a>0$ and $b>0$ such that 
\begin{equation*}
\left\vert f\right\vert \left( x\right) \leq ar\left( x\right) +b\ \ \text{
on } M.
\end{equation*}%
Then for any positive $f$-harmonic function $u$ on $M,$ we have 
\begin{equation*}
\sup_{M}\left\vert \nabla \log u\right\vert ^{2}\leq C\left( n\right) \left(
\Omega \left( u\right) ^{2}+a\Omega \left( u\right) \right),
\end{equation*}
where $C\left( n\right)>0$ is a constant only depending on the dimension $n$
of $M$ and 
\begin{equation*}
\Omega \left( u\right) :=\underset{R\rightarrow \infty }{\limsup}\left\{ 
\frac{1}{R}\sup_{B_{p}\left( R\right) }\log \left( u+1\right) \right\}.
\end{equation*}
\end{proposition}

\begin{proof}[Proof of Proposition \protect\ref{P}]
The proof is based on Yau's argument using the Bochner technique, with a
modification similar to that in \cite{B}. Let $u>0$ be $f$-harmonic. For $%
0<\epsilon <\frac{1}{2},$ define%
\begin{equation*}
h:=\frac{1}{\epsilon }u^{\epsilon }.
\end{equation*}%
Then a direct computation gives $\Delta _{f}h=\left( \epsilon -1\right)
u^{\epsilon -2}\left\vert \nabla u\right\vert ^{2}.$

Let us denote $\sigma :=\left\vert \nabla h\right\vert ^{2}=u^{2\epsilon
-2}\left\vert \nabla u\right\vert ^{2}$. The Bochner formula asserts that 
\begin{gather*}
\frac{1}{2}\Delta _{f}\sigma =\left\vert h_{ij}\right\vert ^{2}+\left\langle
\nabla h,\nabla \left( \Delta _{f}h\right) \right\rangle +Ric_{f}\left(
\nabla h,\nabla h\right) \\
\geq \left\langle \nabla h,\nabla \left( \Delta _{f}h\right) \right\rangle
=\left( \epsilon -1\right) \left\langle \nabla h,\nabla \left( u^{\epsilon
-2}\left\vert \nabla u\right\vert ^{2}\right) \right\rangle \\
=\left( \epsilon -1\right) \left\langle \nabla h,\nabla \left( u^{-\epsilon
}\sigma \right) \right\rangle .
\end{gather*}%
Notice that 
\begin{eqnarray*}
\left\langle \nabla h,\nabla \left( u^{-\epsilon }\sigma \right)
\right\rangle &=&-\epsilon \left\langle \nabla h,\nabla u\right\rangle
u^{-\epsilon -1}\sigma +u^{-\epsilon }\left\langle \nabla h,\nabla \sigma
\right\rangle \\
&=&-\epsilon \left\vert \nabla h\right\vert ^{2}u^{-2\epsilon }\sigma
+u^{-\epsilon }\left\langle \nabla h,\nabla \sigma \right\rangle.
\end{eqnarray*}%
Consequently, 
\begin{equation}
\frac{1}{2}\Delta _{f}\sigma \geq \epsilon \left( 1-\epsilon \right)
u^{-2\epsilon }\sigma ^{2}+\left( \epsilon -1\right) u^{-\epsilon
}\left\langle \nabla h,\nabla \sigma \right\rangle .  \label{a1}
\end{equation}%
We now take a function $\phi :\left[ 0,2R\right] \rightarrow \left[ 0,1%
\right] $ with the following properties: 
\begin{eqnarray*}
\phi &=&1\ \ \text{on}\ \ \left[ 0,R\right] \\
supp\left( \phi \right) &\subseteq &[0,2R) \\
-\frac{c}{R} &\leq &\frac{\phi ^{\prime }}{\sqrt{\phi }}\leq 0 \\
\left\vert \phi ^{\prime \prime }\right\vert &\leq &\frac{c}{R^{2}},
\end{eqnarray*}%
where $c>0$ is a universal constant. We use this function to define a
cut-off on $M$ by taking $\phi \left( x\right) :=\phi \left( r\left(
x\right) \right).$ Define $G:=\phi \sigma .$ Then $G$ is non-negative on $M$
and has compact support in $B_{p}\left( 2R\right) .$ Therefore, it achieves
its maximum at some point $y\in B_{p}\left( 2R\right) .$

Without loss of generality we can assume that $y$ is not in the cut-locus of 
$p.$ So $\phi $ is smooth at $p$. Now at point $y,$ we have%
\begin{equation}
\Delta G\left( y\right) \leq 0\ \ \text{and}\ \ \nabla G\left( y\right) =0.
\label{a2}
\end{equation}%
By (\ref{v2}) together with $\left\vert f \right\vert \left( x\right) \leq
ar\left( x\right) +b,$ we get 
\begin{equation*}
\Delta _{f}r\left( x\right) \leq \frac{n-1+4b}{r}+3a.
\end{equation*}%
So there exists $r_{0}>0$ such that $\Delta _{f}r\left( x\right) \leq 4a$
for $x\in M\backslash B_{p}\left( r_{0}\right) .$ Therefore, for $R>r_{0},$
we have 
\begin{equation*}
\Delta _{f}\phi =\phi ^{\prime }\Delta _{f}r+\phi ^{\prime \prime }\geq
-c\left( \frac{a}{R}+\frac{1}{R^{2}}\right)
\end{equation*}%
and 
\begin{equation*}
\phi ^{-1}\left\vert \nabla \phi \right\vert ^{2}\leq \frac{c}{R^{2}}.
\end{equation*}%
Combining with inequality (\ref{a1}), we find that%
\begin{gather*}
\frac{1}{2}\Delta _{f}G=\frac{1}{2}\phi \Delta _{f}\sigma +\frac{1}{2}\sigma
\Delta _{f}\phi +\left\langle \nabla \phi ,\nabla \sigma \right\rangle \\
\geq \epsilon \left( 1-\epsilon \right) u^{-2\epsilon }\phi
^{-1}G^{2}+\left( \epsilon -1\right) u^{-\epsilon }\left\langle \nabla
h,\nabla \left( \phi ^{-1}G\right) \right\rangle \phi \\
-c\left( \frac{a}{R}+\frac{1}{R^{2}}\right) \phi ^{-1}G+\left\langle \nabla
\phi ,\nabla \left( \phi ^{-1}G\right) \right\rangle \\
=\epsilon \left( 1-\epsilon \right) u^{-2\epsilon }\phi ^{-1}G^{2}+\left(
\epsilon -1\right) u^{-\epsilon }\left\langle \nabla h,\nabla \phi
^{-1}\right\rangle G\phi +\left( \epsilon -1\right) u^{-\epsilon
}\left\langle \nabla h,\nabla G\right\rangle \\
-c\left( \frac{a}{R}+\frac{1}{R^{2}}\right) \phi ^{-1}G-\left\vert \nabla
\phi \right\vert ^{2}\phi ^{-2}G+\phi ^{-1}\left\langle \nabla \phi ,\nabla
G\right\rangle .
\end{gather*}%
After multiplying both sides by $\phi$ and invoking (\ref{a2}), we conclude
that at $y,$%
\begin{equation*}
0\geq \epsilon \left( 1-\epsilon \right) u^{-2\epsilon }G^{2}-\left(
1-\epsilon \right) u^{-\epsilon }\left\vert \nabla h\right\vert \left\vert
\nabla \phi \right\vert G-c\left( \frac{a}{R}+\frac{1}{R^{2}}\right) G.
\end{equation*}%
Note that 
\begin{equation*}
u^{-\epsilon }\left\vert \nabla h\right\vert \left\vert \nabla \phi
\right\vert \leq \frac{c}{R}u^{-\epsilon }\sigma ^{\frac{1}{2}}\phi ^{\frac{1%
}{2}}=\frac{c}{R}\left( u^{-2\epsilon }G\right) ^{\frac{1}{2}}.
\end{equation*}%
This shows at point $y,$ 
\begin{equation*}
\epsilon \left( u^{-2\epsilon }G\right) -\frac{c}{R}\left( u^{-2\epsilon
}G\right) ^{\frac{1}{2}}-c\left( \frac{a}{R}+\frac{1}{R^{2}}\right) \leq 0.
\end{equation*}%
Solving this as a quadratic inequality in $\left( u^{-2\epsilon }G\right) ^{%
\frac{1}{2}}$ we get 
\begin{equation*}
u^{-2\epsilon }\left( y\right) G\left( y\right) \leq \frac{c}{\left(
\epsilon R\right) ^{2}}+\frac{ca}{\epsilon R}.
\end{equation*}%
This proves that 
\begin{gather*}
\sup_{B_{p}\left( R\right) }\left( u^{2\epsilon }\left\vert \nabla \log
u\right\vert ^{2}\right) =\sup_{B_{p}\left( R\right) }\sigma \leq
\sup_{B_{p}\left( 2R\right) }G \\
\leq \left( \frac{c}{\left( \epsilon R\right) ^{2}}+\frac{ca}{\epsilon R}%
\right) \sup_{B_{p}\left( 2R\right) }\left( u^{2\epsilon }\right) .
\end{gather*}

We now observe that if $u$ is globally bounded on $M,$ then the estimate
implies $u$ is constant by letting $R\rightarrow \infty.$ That means the
gradient estimate claimed in the Proposition is automatically true. For
unbounded $u,$ we let 
\begin{equation*}
\epsilon :=\left( 2+\sup_{B_{p}\left( 2R\right) }\log \left( u+1\right)
\right) ^{-1}>0.
\end{equation*}%
Then, 
\begin{equation*}
\sup_{B_{p}\left( 2R\right) }\left( u^{2\epsilon }\right) \leq e^{2}.
\end{equation*}%
So we obtain, for any $R>0$ and $r<R,$ that 
\begin{equation*}
\sup_{B_{p}\left( r\right) }\left( u^{2\epsilon }\left\vert \nabla \log
u\right\vert ^{2}\right) \leq \left( \frac{c}{R}\sup_{B_{p}\left( 2R\right)
}\log \left( u+1\right) \right) ^{2}+\frac{ca}{R}\sup_{B_{p}\left( 2R\right)
}\log \left( u+1\right) +\frac{c\left( a+1\right) }{R}.
\end{equation*}

Since $u$ is unbounded, it is clear that $\epsilon \rightarrow 0$ as $%
R\rightarrow \infty .$ Therefore, after letting $R\rightarrow \infty $ with $%
r$ fixed, we arrive at 
\begin{equation*}
\sup_{B_{p}\left( r\right) }\left\vert \nabla \log u\right\vert ^{2}\leq
C\left( n\right) \left( \Omega \left( u\right) ^{2}+a\Omega \left( u\right)
\right) .
\end{equation*}%
Since $r$ is arbitrary, this proves the Proposition.
\end{proof}

Let us point out that it is possible to prove another version of Proposition %
\ref{P} without any growth assumption on $f.$ It takes the form of $%
\sup_{M}\left\vert \nabla \log u\right\vert ^{2}\leq C\left( \Omega \left(
u\right) ^{2}+\Omega \left( u\right) \right) $ for some constant $C$
depending on the Ricci curvature lower bound and $\sup \left\vert \nabla
f\right\vert $ on $B_{p}\left( 1\right).$ This is because we can use a
different Laplacian comparison theorem from \cite{W}, which does not require 
$f$ to be bounded. It in particular says that a positive $f$-harmonic
function of sub-exponential growth on a complete manifold with $Ric_{f}\geq
0 $ must be a constant.

In the next step, we will establish an upper bound estimate for $\Omega
\left( u\right) $ defined in Proposition \ref{P} by using Moser iteration
argument. First, we will establish a local Sobolev inequality on $M$,
following the arguments in \cite{Bu, HK, G, SC}. Since it will be crucial to
have explicit and accurate dependency of the constants appearing in the
inequality in terms of the growth of $f,$ we provide details of the proof
here.

We now use Lemma \ref{Vol_Ric_Pos} to prove a Neumann Poincar\'{e}
inequality. For this, we follow Buser's proof, see \cite{Bu} (also cf. \cite%
{Ch}, p. 354). There is an alternate proof, see \cite{SC1}, p. 176, which
first establishes a weaker version of Neumann Poincar\'{e} and then uses a
covering argument to prove the strong version. For the proof of Theorem \ref%
{Grad_Est}, the weaker version of Neumann Poincar\'{e} inequality is in fact
sufficient. Recall 
\begin{equation}
A\left( R\right) :=\sup_{x\in B_{p}\left( 3R\right) }\left\vert f\right\vert
\left( x\right) .  \label{a5}
\end{equation}%
In the following, we will suppress $R$ in $A(R)$ and simply call it $A. $

\begin{lemma}
\label{NP} Let $\left( M,g,e^{-f}dv\right) $ be a smooth metric measure
space with $Ric_{f}\geq 0.$ Then for any $x\in B_{p}\left( R\right) $ we
have 
\begin{equation*}
\int_{B_{x}\left( r\right) }\left\vert \varphi -\varphi _{B_{x}\left(
r\right) }\right\vert ^{2}\leq c_{1}e^{c_{2}A}\cdot r^{2}\int_{B_{x}\left(
r\right) }\left\vert \nabla \varphi \right\vert ^{2}
\end{equation*}%
for all $0<r<R$ and $\varphi \in C^{\infty }\left( B_{x}\left( r\right)
\right) ,$ where $\varphi _{B_{x}\left( r\right) }:=V^{-1}\left( B_{x}\left(
r\right) \right) \int_{B_{x}\left( r\right) }\varphi .$ The constants $c_{1}$
and $c_{2}$ depend only on the dimension $n.$
\end{lemma}

\begin{proof}[Proof of Lemma \protect\ref{NP}]
First, we show that Lemma \ref{Vol_Ric_Pos} and the argument in \cite{Bu}
imply a lower bound on the isoperimetric constant. Let $\Gamma $ be a smooth
hypersurface in $B_{x}\left( r\right) $ with $\bar{\Gamma}$ imbedded in $%
\overline{B_{x}\left( r\right) }.$ Let $D_{1}$ and $D_{2}$ be disjoint open
subsets in $B_{x}\left( r\right) $ such that $D_{1}\cup D_{2}=B_{x}\left(
r\right) \backslash \Gamma .$ We will show that 
\begin{equation}
\min \left\{ V\left( D_{1}\right) ,V\left( D_{2}\right) \right\} \leq
rc_{1}e^{c_{2}A}A\left( \Gamma \right) .  \label{a3}
\end{equation}%
Then, by Cheeger's theorem, the inequality (\ref{a3}) gives the claimed
Neumann Poincar\'{e} inequality.

We fix $D_{1}$ so that $V\left( D_{1}\cap B_{x}\left( \frac{r}{2}\right)
\right) \leq \frac{1}{2}V\left( B_{x}\left( \frac{r}{2}\right) \right) .$
For a fixed $\alpha \in \left( 0,1\right) $ to be chosen later, consider
first the case when $V\left( D_{1}\cap B_{x}\left( \frac{r}{2}\right)
\right) \leq \alpha V\left( D_{1}\right) .$ We denote by $C\left( x\right) $
the cut locus of $x.$ For $y\in D_{1}\backslash C\left( x\right),$ let $%
y^{\ast }$ be the first intersection point of $yx$, the unique minimizing
geodesic from $y$ to $x,$ with $\Gamma .$ In the case $yx$ does not
intersect $\Gamma ,$ we set $y^{\ast }=x.$

Define 
\begin{eqnarray*}
\mathcal{A}_{1} &:&=\left\{ y\in D_{1}\backslash \left( C\left( x\right)
\cup B_{x}\left( \frac{r}{2}\right) \right) :\ \ y^{\ast }\notin B_{x}\left( 
\frac{r}{4}\right) \right\} \\
\mathcal{A}_{2} &:&=\left\{ y\in D_{1}\backslash \left( C\left( x\right)
\cup B_{x}\left( \frac{r}{2}\right) \right) :\ \ y^{\ast }\in B_{x}\left( 
\frac{r}{4}\right) \right\} \\
\mathcal{A}_{3} &:&=\left( B_{x}\left( \frac{r}{2}\right) \backslash
B_{x}\left( \frac{r}{4}\right) \right) \cap \left( \cup _{y\in \mathcal{A}%
_{2}}rod(y)\right) ,
\end{eqnarray*}%
where $rod(y):=\left\{ \exp _{x}\left( \tau \xi \right) :\ \ \frac{r}{4}%
<\tau <s\right\} $ for $y=\exp _{x}\left( s\xi \right) $ with $\xi \in
S_{x}M $ and $0<s<r .$

From now on, we will use $c_{1}$ and $c_{2}$ to denote constants depending
only on dimension $n.$ By Lemma \ref{Vol_Ric_Pos} we have $\frac{V\left( 
\mathcal{A}_{2}\right) }{V\left( \mathcal{A}_{3}\right) }\leq
c_{1}e^{c_{2}A}.$ Note that $V\left( \mathcal{A}_{1}\right) +V\left( 
\mathcal{A}_{2}\right) =V\left( D_{1}\backslash B_{x}\left( \frac{r}{2}%
\right) \right) \geq \left( 1-\alpha \right) V\left( D_{1}\right) $ and $%
V\left( \mathcal{A}_{3}\right) \leq V\left( D_{1}\cap B_{x}\left( \frac{r}{2}%
\right) \right) \leq \alpha V\left( D_{1}\right).$ Thus, 
\begin{eqnarray*}
\left( 1-\alpha \right) V\left( D_{1}\right) &\leq &V\left( \mathcal{A}%
_{1}\right) +V\left( \mathcal{A}_{2}\right) \leq V\left( \mathcal{A}%
_{1}\right) +c_{1}e^{c_{2}A}V\left( \mathcal{A}_{3}\right) \\
&\leq &V\left( \mathcal{A}_{1}\right) +\alpha c_{1}e^{c_{2}A}V\left(
D_{1}\right),
\end{eqnarray*}%
or%
\begin{equation}
V\left( \mathcal{A}_{1}\right) \geq \left( 1-\alpha c_{1}e^{c_{2}A}\right)
V\left( D_{1}\right) .  \label{a4}
\end{equation}%
By setting 
\begin{equation}
\alpha :=\frac{1}{2}\left( c_{1}\right) ^{-1}e^{-c_{2}A} ,  \label{a6}
\end{equation}%
we conclude $V\left( \mathcal{A}_{1}\right) \geq \frac{1}{2}V\left(
D_{1}\right) .$ Moreover, Lemma \ref{Vol_Ric_Pos} implies that 
\begin{equation*}
\frac{A\left( \Gamma \right) }{V\left( \mathcal{A}_{1}\right) }\geq \left(
c_{3}\right) ^{-1}e^{-c_{4}A}r^{-1},
\end{equation*}%
where constants $c_{3}$ and $c_{4}$ depending only on $n.$

This, together with (\ref{a4}), proves the Lemma in the case when $V\left(
D_{1}\cap B_{x}\left( \frac{r}{2}\right) \right) \leq \alpha V\left(
D_{1}\right) ,$ where $\alpha $ is given by (\ref{a6}).

Now we assume that $V\left( D_{1}\cap B_{x}\left( \frac{r}{2}\right) \right)
\geq \alpha V\left( D_{1}\right) $ and finish the proof of the Lemma. We
need the following general fact, see \cite{Bu}. Set $W_{0}:=D_{1}\cap
B_{x}\left( \frac{r}{2}\right) $ and $W_{1}:=D_{2}\cap B_{x}\left( \frac{r}{2%
}\right) $ or vice versa $W_{0}:=D_{2}\cap B_{x}\left( \frac{r}{2}\right) $
and $W_{1}:=D_{1}\cap B_{x}\left( \frac{r}{2}\right) .$ Then, for at least
one of the two choices of $\left\{ W_{0},W_{1}\right\},$ there exist a point 
$w_{0}\in W_{0}$ and a measurable set $\mathcal{W}_{1}\subset W_{1}$ so that 
$V\left( \mathcal{W}_{1}\right) \geq \frac{1}{2}V\left( W_{1}\right).$
Moreover, for each $y\in \mathcal{W}_{1},$ the minimizing geodesic $yw_{0}$
from $y$ to $w_{0}$ intersects $\Gamma $ and the first intersection point $%
y^{\ast }$ satisfies $d\left( y,y^{\ast }\right) \leq d\left(
y^{\ast},w_{0}\right) .$

Observe that $\alpha V\left( D_{1}\right) \leq V\left( D_{1}\cap B_{x}\left( 
\frac{r}{2}\right) \right) \leq \frac{1}{2}V\left( B_{x}\left( \frac{r}{2}%
\right) \right) .$ Therefore, $\alpha V\left( D_{1}\right) \leq V\left(
D_{2}\cap B_{x}\left( \frac{r}{2}\right) \right) ,$ too. So, regardless of
how $\mathcal{W}_{1}$ is picked, we have $\alpha V\left( D_{1}\right) \leq
2V\left( \mathcal{W}_{1}\right) .$ Now, to establish the Lemma, we need only
to show $\frac{A\left( \Gamma \right) }{V\left( \mathcal{W}_{1}\right) }\geq
\left( c_{3}\right) ^{-1}e^{-c_{4}A}r^{-1}$ for some $c_{3}$ and $c_{4} $
depending only on $n.$ For this, we use polar coordinates at $w_{0}.$ For $%
y\in \mathcal{W}_{1}$, write $y=\exp _{w_{0}}\left( t_{0}\xi \right).$ Let $%
y^{\ast }$ be the first intersecting point of $yw_{0}$ with $\Gamma .$
Define $t_1$ by $y^{\ast }=\exp _{w_{0}}\left( t_{1}\xi \right) .$ From the
choice of $w_{0}$ and $\mathcal{W}_{1},$ we know $t_{1}\geq \frac{1}{2}%
t_{0}. $ Let $t_{2}$ be the maximal $t$ so that $\exp _{w_{0}}\left( t\xi
\right) \in \mathcal{W}_{1}\backslash C\left( w_{0}\right) .$ Clearly, $%
t_{2}\leq 2t_{1}.$ From Lemma \ref{Vol_Ric_Pos} we conclude that 
\begin{eqnarray*}
\int_{t_{1}}^{t_{2}}J\left( w_{0},t,\xi \right) dt &\leq &\left(
t_{2}-t_{1}\right) J\left( w_{0},t_{1},\xi \right) c_{1}e^{c_{2}A} \\
&\leq &rJ\left( w_{0},t_{1},\xi \right) c_{1}e^{c_{2}A}.
\end{eqnarray*}%
The desired result then follows after integrating in $\xi $. This proves the
Lemma.
\end{proof}

Combining Lemma \ref{Vol_Ric_Pos}, Lemma \ref{NP} and the argument in \cite%
{HK}, we get a local Neumann Sobolev inequality of the following form.

\begin{lemma}
\label{NS} Let $\left( M,g,e^{-f}dv\right) $ be a smooth metric measure
space with $Ric_{f}\geq 0.$ Then there exist constants $\nu >2,$ $c_{1}$ and 
$c_{2},$ all depending only on $n$ such that 
\begin{equation*}
\left( \int_{B_{p}\left( R\right) }\left\vert \varphi -\varphi _{B_{p}\left(
R\right) }\right\vert ^{\frac{2\nu }{\nu -2}}\right) ^{\frac{\nu -2}{\nu }%
}\leq c_{1}e^{c_{2}A}\frac{R^{2}}{V\left( B_{p}\left( R\right) \right) ^{%
\frac{2}{\nu }}}\int_{B_{p}\left( R\right) }\left\vert \nabla \varphi
\right\vert ^{2}\ 
\end{equation*}%
for $\ \varphi \in C^{\infty }\left( B_{p}\left( R\right) \right) ,$ where $%
\varphi _{B_{p}\left( R\right) }:=V^{-1}\left( B_{p}\left( R\right) \right)
\int_{B_{p}\left( R\right) }\varphi .$
\end{lemma}

\begin{proof}[Proof of Lemma \protect\ref{NS}]
For $y\in B_{p}\left( R\right),$ let $\gamma \left( t\right) $ be a
minimizing geodesic from $p$ to $y$ such that $\gamma \left( 0\right) =p$
and $\gamma \left( L\right) =y.$ Define $y_{0}:=p$ and $y_{i}:=\gamma \left(
\sum_{j=1}^{i}\frac{R}{2^{j}}\right) $ for $1\leq i\leq i_{0},$ where $i_{0}$
is the largest integer $i$ so that $\sum_{j=1}^{i}\frac{R}{2^{j}}<L.$ Define
also $B_{i}:=B_{y_{i}}\left( \frac{R}{2^{i+1}}\right) $ for $i<i_{0}$ and $%
B_{i}:=B_{y}\left( \frac{R}{2^{i+1}}\right) $ for $i\geq i_{0}.$

Let $\varphi _{B_{i}}:=V^{-1}\left( B_{i}\right) \int_{B_{i}}\varphi .$ Then 
$\lim_{i\rightarrow \infty }\varphi _{B_{i}}=\varphi \left( y\right) .$
Thus, we have 
\begin{equation*}
\left\vert \varphi _{B_{0}}-\varphi \left( y\right) \right\vert \leq
\sum_{i\geq 0}\left\vert \varphi _{B_{i}}-\varphi _{B_{i+1}}\right\vert \leq
\sum_{i\geq 0}\left( \left\vert \varphi _{B_{i}}-\varphi _{D_{i}}\right\vert
+\left\vert \varphi _{D_{i}}-\varphi _{B_{i+1}}\right\vert \right) .
\end{equation*}%
Here, $D_{i}:=B_{z_{i}}\left( \frac{R}{2^{i+3}}\right) \subset B_{i}\cap
B_{i+1},$ and $z_{i}:=\gamma \left( \sum_{j=1}^{i}\frac{R}{2^{j}}+\frac{3R}{%
2^{i+3}}\right) .$

By Lemma \ref{Vol_Ric_Pos} it is easy to see that 
\begin{equation*}
\left\vert \varphi _{B_{i}}-\varphi _{D_{i}}\right\vert \leq V\left(
D_{i}\right) ^{-1}\int_{B_{i}}\left\vert \varphi -\varphi
_{B_{i}}\right\vert \leq c_{1}e^{c_{2}A}V\left( B_{i}\right)
^{-1}\int_{B_{i}}\left\vert \varphi -\varphi _{B_{i}}\right\vert .
\end{equation*}%
Note that a similar bound for $\left\vert \varphi _{D_{i}}-\varphi
_{B_{i+1}}\right\vert $ also holds. So we conclude that 
\begin{eqnarray*}
\left\vert \varphi _{B_{0}}-\varphi \left( y\right) \right\vert &\leq
&c_{1}e^{c_{2}A}\sum_{i\geq 0}V\left( B_{i}\right)
^{-1}\int_{B_{i}}\left\vert \varphi -\varphi _{B_{i}}\right\vert \\
&\leq &c_{1}e^{c_{2}A}\sum_{i\geq 0}\left( V\left( B_{i}\right)
^{-1}\int_{B_{i}}\left\vert \varphi -\varphi _{B_{i}}\right\vert ^{2}\right)
^{\frac{1}{2}} \\
&\leq &c_{1}e^{c_{2}A}\sum_{i\geq 0}\frac{R}{2^{i+1}}\left( V\left(
B_{i}\right) ^{-1}\int_{B_{i}}\left\vert \nabla \varphi \right\vert
^{2}\right) ^{\frac{1}{2}},
\end{eqnarray*}%
where in the second line we have use the Cauchy-Schwarz inequality and in
the last line we have use Lemma \ref{NP}. On the other hand, 
\begin{equation*}
\left\vert \varphi _{B_{0}}-\varphi \left( y\right) \right\vert =cR^{-\frac{1%
}{2}}\sum_{i\geq 0}\left( \frac{R}{2^{i+1}}\right) ^{\frac{1}{2}}\left\vert
\varphi _{B_{0}}-\varphi \left( y\right) \right\vert ,
\end{equation*}%
where $c$ is a universal constant. So for $R_{i}:=\frac{R}{2^{i+1}},$ we
have 
\begin{equation*}
\sum_{i\geq 0}\left( R_{i}\right) ^{\frac{1}{2}}\left\vert \varphi
_{B_{0}}-\varphi \left( y\right) \right\vert \leq c_{1}e^{c_{2}A}R^{\frac{1}{%
2}}\sum_{i\geq 0}R_{i}\left( V\left( B_{i}\right)
^{-1}\int_{B_{i}}\left\vert \nabla \varphi \right\vert ^{2}\right) ^{\frac{1%
}{2}}.
\end{equation*}%
Hence, there exists an $i$ (depending on $y$) so that%
\begin{equation*}
\left\vert \varphi _{B_{0}}-\varphi \left( y\right) \right\vert ^{2}\leq
c_{1}e^{c_{2}A}\left( RR_{i}\right) \frac{1}{V\left( B_{i}\right) }%
\int_{B_{i}}\left\vert \nabla \varphi \right\vert ^{2}.
\end{equation*}%
Since $B_{i}\subset B_{y}\left( 3R_{i}\right),$ it follows that for each $%
y\in B_{p}\left( R\right) $ there exists $r_{y}>0$ so that%
\begin{equation}
\left\vert \varphi _{B_{0}}-\varphi \left( y\right) \right\vert ^{2}\leq
c_{1}e^{c_{2}A}\left( Rr_{y}\right) V\left( B_{y}\left( r_{y}\right) \right)
^{-1}\int_{B_{y}\left( r_{y}\right) \cap B_{p}\left( R\right) }\left\vert
\nabla \varphi \right\vert ^{2}.  \label{a7}
\end{equation}%
According to Lemma \ref{Vol_Ric_Pos}, 
\begin{equation}
\frac{V\left( B_{p}\left( R\right) \right) }{V\left( B_{y}\left(
r_{y}\right) \right) }\leq \frac{V\left( B_{y}\left( 2R\right) \right) }{%
V\left( B_{y}\left( r_{y}\right) \right) }\leq c_{1}e^{c_{2}A}\left( \frac{R%
}{r_{y}}\right) ^{n},  \label{a8}
\end{equation}%
Solving $r_{y}$ from (\ref{a8}) and plugging into (\ref{a7}) then gives 
\begin{equation}
\left\vert \varphi _{B_{0}}-\varphi \left( y\right) \right\vert ^{2}\leq
c_{1}e^{c_{2}A}R^{2}V\left( B_{p}\left( R\right) \right) ^{-\frac{1}{n}%
}V\left( B_{y}\left( r_{y}\right) \right) ^{\frac{1}{n}-1}\int_{B_{y}\left(
r_{y}\right) \cap B_{p}\left( R\right) }\left\vert \nabla \varphi
\right\vert ^{2}.  \label{a9}
\end{equation}%
We now define $A_{t}:=\left\{ y\in B_{p}\left( R\right) :\ \left\vert
\varphi _{B_{0}}-\varphi \left( y\right) \right\vert \geq t\right\}.$
Applying the Vitali covering Lemma, we find a countable disjoint collection $%
\left\{ B_{i}\left( r_{i}\right) \right\} _{i\in I}$ of balls from $\left\{
B_{y}\left( r_{y}\right) :\ y\in A_{t}\right\} $ such that for any $y\in
A_{t},$ there exists $i\in I$ such that $B_{i}\left( r_{i}\right) \cap
B_{y}\left( r_{y}\right) \neq \phi $ and $B_{y}\left( r_{y}\right) \subset
B_{i}\left( 3r_{i}\right) .$ Then it follows, by Lemma \ref{Vol_Ric_Pos} and
(\ref{a9}), that 
\begin{eqnarray*}
V\left( A_{t}\right) ^{1-\frac{1}{n}} &\leq &c_{1}e^{c_{2}A}\sum_{i\in
I}V\left( B_{i}\right) ^{1-\frac{1}{n}} \\
&\leq &c_{1}e^{c_{2}A}\frac{R^{2}}{t^{2}}V\left( B_{p}\left( R\right)
\right) ^{-\frac{1}{n}}\sum_{i\in I}\int_{B_{i}\left( r_{i}\right) \cap
B_{p}\left( R\right) }\left\vert \nabla \varphi \right\vert ^{2} \\
&=&c_{1}e^{c_{2}A}\frac{R^{2}}{t^{2}}V\left( B_{p}\left( R\right) \right) ^{-%
\frac{1}{n}}\int_{B_{p}\left( R\right) }\left\vert \nabla \varphi
\right\vert ^{2}.
\end{eqnarray*}%
This may be rewritten as 
\begin{equation}
V\left( A_{t}\right) \leq t^{-\frac{2n}{n-1}}B  \label{a10}
\end{equation}%
with $B:=c_{1}e^{c_{2}A}R^{\frac{2n}{n-1}}V\left( B_{p}\left( R\right)
\right) ^{-\frac{1}{n-1}}\left( \int_{B_{p}\left( R\right) }\left\vert
\nabla \varphi \right\vert ^{2}\right) ^{\frac{n}{n-1}}.$ Now, for any $%
\frac{2n}{n-1}>q>2,$ we have 
\begin{eqnarray*}
\int_{B_{p}\left( R\right) }\left\vert \varphi -\varphi _{B_{0}}\right\vert
^{q} &=&q\int_{0}^{\infty }t^{q-1}V\left( A_{t}\right) dt \\
&=&q\int_{0}^{T}t^{q-1}V\left( A_{t}\right) dt+q\int_{T}^{\infty
}t^{q-1}V\left( A_{t}\right) dt \\
&\leq &T^{q}V\left( B_{p}\left( R\right) \right) +\frac{q}{\frac{2n}{n-1}-q}%
T^{q-\frac{2n}{n-1}}B,
\end{eqnarray*}%
where we have used (\ref{a10}) to bound the second integral in the second
line. Choosing $q=\frac{2n-1}{n-1}$ and $T:=\left( R^{2}\frac{1}{V\left(
B_{p}\left( R\right) \right) }\int_{B_{p}\left( R\right) }\left\vert \nabla
\varphi \right\vert ^{2}\right) ^{\frac{1}{2}},$ we get 
\begin{equation*}
\left( \int_{B_{p}\left( R\right) }\left\vert \varphi -\varphi
_{B_{0}}\right\vert ^{\frac{2\nu }{\nu -2}}\right) ^{\frac{\nu -2}{\nu }%
}\leq c_{1}e^{c_{2}A}\frac{R^{2}}{V\left( B_{p}\left( R\right) \right) ^{%
\frac{2}{\nu }}}\int_{B_{p}\left( R\right) }\left\vert \nabla \varphi
\right\vert ^{2},
\end{equation*}%
where $\nu :=\frac{2q}{q-2}=4n-2.$ This proves the Theorem.
\end{proof}

We are now ready to prove Theorem \ref{Grad_Est}.

\begin{proof}[Proof of Theorem \protect\ref{Grad_Est}]
Let $u$ be a positive solution to $\Delta _{f} u=0.$ Applying the Moser
iteration scheme to the equation $\Delta _{f}u=0$ with the help of Lemma \ref%
{Vol_Ric_Pos} and Lemma \ref{NS}, we obtain 
\begin{equation}
\sup_{B_{p}\left( \frac{1}{2}R\right) }u\leq \frac{c_{1}e^{c_{2}A}}{V\left(
B_{p}\left( R\right) \right) }\int_{B_{p}\left( R\right) }u,  \label{a11}
\end{equation}%
where $c_{1}$ and $c_{2}$ are constants depending only on $n$. For this,
notice that the Neumann Sobolev inequality, Lemma \ref{NS}, also holds true
when integrals are with respect to the measure $e^{-f}dv,$ without changing
the nature of the dependency of the constants on $A.$

Now we start using the assumption on $f$ that $\left\vert f\right\vert
\left( x\right) \leq ar\left( x\right) +b$ on $M.$ Applying (\ref{v2}) we
obtain, for any $r>0,$ that $\Delta _{f}r\left( x\right) \leq \frac{n-1+4b}{r%
}+3a.$ Therefore, we can find $r_{0}>0$ so that 
\begin{equation}
\Delta _{f}r\left( x\right) \leq 4a\ \ \ \text{for any }x\in M\backslash
B_{p}\left( r_{0}\right) .  \label{a12}
\end{equation}%
So for $r>r_{0},$ 
\begin{gather*}
4a\int_{B_{p}\left( r\right) \backslash B_{p}\left( r_{0}\right)
}ue^{-f}\geq \int_{B_{p}\left( r\right) \backslash B_{p}\left( r_{0}\right)
}u\left( \Delta _{f}r\right) e^{-f} \\
=r_{0}\int_{\partial B_{p}\left( r_{0}\right) }\left\langle \nabla u,\nabla
r\right\rangle e^{-f}-r\int_{\partial B_{p}\left( r\right) }\left\langle
\nabla u,\nabla r\right\rangle e^{-f}+\int_{\partial B_{p}\left( r\right)
}ue^{-f}-\int_{\partial B_{p}\left( r_{0}\right) }ue^{-f} \\
=\int_{\partial B_{p}\left( r\right) }ue^{-f}-\int_{\partial B_{p}\left(
r_{0}\right) }ue^{-f},
\end{gather*}%
where we have used the fact that $\Delta _{f}u=0$ and 
\begin{equation*}
\int_{\partial B_{p}\left( r\right) }\left\langle \nabla u,\nabla
r\right\rangle e^{-f}=\int_{B_{p}\left( r\right) }\left( \Delta _{f}u\right)
e^{-f}=0.
\end{equation*}%
Denote by 
\begin{eqnarray*}
U\left( r\right) &:&=\int_{B_{p}\left( r\right) \backslash B_{p}\left(
r_{0}\right) }ue^{-f}\text{ \ \ \ and} \\
C_{0} &:&=\int_{\partial B_{p}\left( r_{0}\right) }ue^{-f}.
\end{eqnarray*}%
Then the preceding inequality implies that for $r>r_{0},$ 
\begin{equation*}
U^{\prime }\left( r\right) \leq 4aU\left( r\right) +C_{0}.
\end{equation*}%
After integrating from $r_{0}$ to $R>r_{0},$ we obtain 
\begin{equation*}
U\left( R\right) \leq C_{1}e^{4aR}
\end{equation*}%
with $C_{1}:=U\left( r_{0}\right) +\frac{1}{a}C_{0}.$ Consequently, we have 
\begin{equation*}
\int_{B_{p}\left( R\right) }u\leq C_{2}e^{5aR}.
\end{equation*}%
Plugging this into (\ref{a11}), we conclude that 
\begin{equation*}
\sup_{B_{p}\left( \frac{1}{2}R\right) }u\leq C_{3}e^{c\left( n\right) aR},
\end{equation*}%
where $c\left( n\right) $ depends only on dimension $n$ and $C_{3}$ is
independent of $R$. This shows that%
\begin{equation*}
\Omega \left( u\right) :=\underset{R\rightarrow \infty }{\limsup}\left\{ 
\frac{1}{R}\sup_{B_{p}\left( R\right) }\log \left( u+1\right) \right\} \leq
c\left( n\right) a,
\end{equation*}%
where $c\left( n\right) $ depends only on $n.$ By Proposition \ref{P}, we
conclude 
\begin{equation*}
\sup_{M}\left\vert \nabla \log u\right\vert \leq C\left( n\right) a
\end{equation*}%
with $C\left( n\right) $ being a constant depending only on $n,$ as claimed
in Theorem \ref{Grad_Est}. It is evident that if $f$ has sublinear growth,
then $a$ can be taken as close to zero as we wish, thus proving that $u$
must be constant.
\end{proof}

We now turn to polynomial growth $f$-harmonic functions.

\begin{theorem}
\bigskip \label{Sublin} Let $\left( M,g,e^{-f}dv\right) $ be a complete
noncompact smooth metric measure space with $Ric_{f}\geq 0$ and $f$ bounded.
Then a sublinear growth $f$-harmonic function on $M$ must be a constant.
\end{theorem}

\begin{proof}[Proof of Theorem \protect\ref{Sublin}]
Let $w$ be an $f$-harmonic function so that 
\begin{equation}
\lim_{x\rightarrow \infty }\frac{\left\vert w \right\vert\left( x\right) }{%
r\left( x\right) }=0.  \label{b1}
\end{equation}%
The Bochner formula asserts that 
\begin{equation*}
\frac{1}{2}\Delta _{f}\left\vert \nabla w\right\vert ^{2}=\left\vert
w_{ij}\right\vert ^{2}+Ric_{f}\left( \nabla w,\nabla w\right) \geq 0.
\end{equation*}%
So $\left\vert \nabla w\right\vert ^{2}$ is $f$-subharmonic. Applying the
Moser iteration scheme, we obtain a mean value inequality of the form 
\begin{equation}
\sup_{B_{p}\left( \frac{1}{2}R\right) }\left\vert \nabla w\right\vert
^{2}\leq \frac{C}{V\left( B_{p}\left( R\right) \right) }\int_{B_{p}\left(
R\right) }\left\vert \nabla w\right\vert ^{2}e^{-f}  \label{b2}
\end{equation}%
for some constant $C$ depending on $n$ and $\sup \left\vert f\right\vert .$
Note that now the constant $A$ in Lemma \ref{Vol_Ric_Pos} and Lemma \ref{NS}
is independent of $R$ as $f$ is assumed to be bounded.

We now choose a cut-off $\phi $ such that $\phi =1$ on $B_{p}\left( R\right)
,$ $\phi =0$ on $M\backslash B_{p}\left( 2R\right) $ and $\left\vert \nabla
\phi \right\vert \leq \frac{C}{R}.$ Integrating by parts and using $\Delta
_{f}w=0$ we get 
\begin{eqnarray*}
\int_{M}\left\vert \nabla w\right\vert ^{2}\phi ^{2}e^{-f}
&=&-2\int_{M}w\phi \left\langle \nabla w,\nabla \phi \right\rangle e^{-f} \\
&\leq &2\int_{M}\left\vert w\right\vert \phi \left\vert \left\langle \nabla
w,\nabla \phi \right\rangle \right\vert e^{-f} \\
&\leq &\frac{1}{2}\int_{M}\left\vert \nabla w\right\vert ^{2}\phi
^{2}e^{-f}+2\int_{M}w^{2}\left\vert \nabla \phi \right\vert ^{2}e^{-f}.
\end{eqnarray*}%
This shows that%
\begin{eqnarray*}
\int_{B_{p}\left( R\right) }\left\vert \nabla w\right\vert ^{2}e^{-f} &\leq
&4\int_{M}w^{2}\left\vert \nabla \phi \right\vert ^{2}e^{-f}\leq \frac{C}{%
R^{2}}\int_{B_{p}\left( 2R\right) \backslash B_{p}\left( R\right)
}w^{2}e^{-f} \\
&\leq &\frac{C}{R^{2}}\left( \sup_{B_{p}\left( 2R\right) }w^{2}\right)
V\left( B_{p}\left( 2R\right) \right) \\
&\leq &\frac{C}{R^{2}}\left( \sup_{B_{p}\left( 2R\right) }w^{2}\right)
V\left( B_{p}\left( R\right) \right) ,
\end{eqnarray*}%
where in the last line we have used Lemma \ref{Vol_Ric_Pos}.

Together with (\ref{b1}) we obtain 
\begin{equation*}
\lim_{R\rightarrow \infty }\frac{1}{V\left( B_{p}\left( R\right) \right) }%
\int_{B_{p}\left( R\right) }\left\vert \nabla w\right\vert ^{2}e^{-f}=0.
\end{equation*}%
So by (\ref{b2}), $\left\vert \nabla w\right\vert =0$ on $M.$ The Theorem is
proved.
\end{proof}

Our next result is a dimension estimate for the space of polynomial growth $%
f $-harmonic functions. For this, we use again Moser iteration and Sobolev
inequality. However, the situation here is considerably easier as $f$ is
assumed to be bounded.

\begin{theorem}
\label{Polynomial} Let $\left( M,g,e^{-f}dv\right) $ be a complete
noncompact smooth metric measure space with $Ric_{f}\geq 0$ and $f$ bounded.
Then there exists a constant $\mu >0$ such that 
\begin{equation*}
\dim \mathcal{H}^{d}\left( M\right) \leq Cd^{\mu }\ \ \text{for any\ }d\geq
1.
\end{equation*}%
Moreover, we have the sharp estimate 
\begin{equation*}
\dim \mathcal{H}^{1}\left( M\right) \leq n+1.
\end{equation*}
\end{theorem}

\begin{proof}[Proof of Theorem \protect\ref{Polynomial}]
We first establish the second result about the dimension of the space of
linear growth $f$-harmonic functions. Consider $u$ an $f$-harmonic function
with linear growth, i.e., $\Delta _{f}u=0$ and $\left\vert u\right\vert
\left( x\right) \leq C\left( r\left( x\right) +1\right) $ on $M.$ Using the
argument in the proof of Theorem \ref{Sublin} we want to claim that $%
\left\vert \nabla u\right\vert $ is bounded on $M.$ Indeed, since $f$ is
bounded, (\ref{b2}) is true here, too. Now the reverse Poincar\'e inequality
for $f$-harmonic function yields 
\begin{equation*}
\frac{1}{V\left( B_{p}\left( R\right) \right) }\int_{B_{p}\left( R\right)
}\left\vert \nabla u\right\vert ^{2}e^{-f}\leq \frac{C}{R^{2}}%
\sup_{B_{p}\left( 2R\right) }u^{2}\leq C.
\end{equation*}%
By (\ref{b2}) this shows that $\left\vert \nabla u\right\vert $ is bounded
on $M.$

We now prove a mean value theorem at infinity of the following form. For any
bounded positive $f$-subharmonic function $v$ we have%
\begin{equation}
\lim_{R\rightarrow \infty }\frac{1}{V_{f}\left( B_{p}\left( R\right) \right) 
}\int_{B_{p}\left( R\right) }ve^{-f}=\sup_{M}v.  \label{f1}
\end{equation}%
When $f$ is constant this was first established by P. Li by a heat equation
method. Here, we follow the argument in \cite{CCM} which uses a monotonicity
formula.

Let $w=\sup_{M}v-v,$ which is a positive function that satisfies 
\begin{equation*}
\Delta _{f}w\leq 0\ \ \text{and\ \ }\inf_{M}w=0.
\end{equation*}%
To prove (\ref{f1}) we show instead 
\begin{equation}
\lim_{R\rightarrow \infty }\frac{1}{V_{f}\left( B_{p}\left( R\right) \right) 
}\int_{B_{p}\left( R\right) }we^{-f}=0.  \label{f2}
\end{equation}%
Let $h_{R}$ solve $\Delta _{f}h_{R}=0$ in $B_{p}\left( R\right) $ with $%
h_{R}=w$ on $\partial B_{p}\left( R\right) .$ By the maximum principle, $%
h_{R}$ is positive and uniformly bounded. Moreover, since $\inf_{M}w=0,$ for
any $\varepsilon >0$ there exists $R_{\varepsilon }>0$ such that 
\begin{equation*}
\inf_{B_{p}\left( R\right) }w<\varepsilon \ \ \text{for any }%
R>R_{\varepsilon }.
\end{equation*}%
Again, by the maximum principle, it follows that 
\begin{equation*}
\inf_{B_{p}\left( R\right) }h_{R}<\varepsilon \text{\ \ \ for any }%
R>R_{\varepsilon }.
\end{equation*}%
Notice the following Harnack inequality holds. 
\begin{equation*}
\sup_{B_{p}\left( \frac{1}{2}R\right) }h_{R}\leq C\inf_{B_{p}\left( \frac{1}{%
2}R\right) }h_{R},
\end{equation*}%
where $C$ depends only on $n$ and $\sup_{M}\left\vert f\right\vert .$
Indeed, this follows from Lemma \ref{Vol_Ric_Pos}, Lemma \ref{NP} and Lemma %
\ref{NS} by the Moser iteration argument as in \cite{SC1}, Chapter II.

We therefore conclude that 
\begin{equation}
\sup_{B_{p}\left( \frac{1}{2}R\right) }h_{R}<C\varepsilon \ \ \ \text{for\ \
any }R>2R_{\varepsilon }.  \label{f3}
\end{equation}%
Furthermore, for $R>r>0,$ we have%
\begin{eqnarray*}
0 &=&\int_{B_{p}\left( r\right) }\left( \Delta _{f}h_{R}\right)
e^{-f}=\int_{\partial B_{p}\left( r\right) }\frac{\partial h_{R}}{\partial r}%
e^{-f} \\
&=&\frac{\partial }{\partial r}\int_{\partial B_{p}\left( r\right)
}h_{R}e^{-f}-\int_{\partial B_{p}\left( r\right) }h_{R}\Delta _{f}\left(
r\right) e^{-f} \\
&\geq &\frac{\partial }{\partial r}\int_{\partial B_{p}\left( r\right)
}h_{R}e^{-f}-\frac{C}{r}\int_{\partial B_{p}\left( r\right) }h_{R}e^{-f},
\end{eqnarray*}%
where in the last line we have used (\ref{v2}) and the fact that $f$ is
bounded. This shows that $\log \left( \frac{1}{r^{C}}\int_{\partial
B_{p}\left( r\right) }h_{R}e^{-f}\right) $ is decreasing as a function of $r$
for $0<r<R.$ In particular, it shows that 
\begin{equation*}
\int_{\partial B_{p}\left( R\right) }h_{R}e^{-f}\leq C\int_{\partial
B_{p}\left( \frac{1}{2}R\right) }h_{R}e^{-f}
\end{equation*}%
for a constant $C$ depending only on $n$ and $\sup_{M}\left\vert
f\right\vert .$

So for $R>2R_{\varepsilon },$ 
\begin{eqnarray*}
\int_{\partial B_{p}\left( R\right) }we^{-f} &=&\int_{\partial B_{p}\left(
R\right) }h_{R}e^{-f}\leq C\int_{\partial B_{p}\left( \frac{1}{2}R\right)
}h_{R}e^{-f} \\
&\leq &C\varepsilon A_{f}\left( \partial B_{p}\left( \frac{1}{2}R\right)
\right) ,
\end{eqnarray*}%
where in the second line we have used (\ref{f3}). Since this inequality is
true for all $R>2R_{\varepsilon },$ integrating this from $2R_{\varepsilon }$
to $R$ gives 
\begin{eqnarray*}
\int_{B_{p}\left( R\right) \backslash B_{p}\left( 2R_{\varepsilon }\right)
}we^{-f} &\leq &C\varepsilon \int_{2R_{\varepsilon }}^{R}A_{f}\left(
\partial B_{p}\left( \frac{1}{2}t\right) \right) dt \\
&\leq &C\varepsilon V_{f}\left( B_{p}\left( \frac{1}{2}R\right) \right) .
\end{eqnarray*}%
Hence, for $R$ sufficiently large, 
\begin{gather*}
\frac{1}{V_{f}\left( B_{p}\left( R\right) \right) }\int_{B_{p}\left(
R\right) }we^{-f}=\frac{1}{V_{f}\left( B_{p}\left( R\right) \right) }%
\int_{B_{p}\left( R\right) \backslash B_{p}\left( 2R_{\varepsilon }\right)
}we^{-f} \\
+\frac{1}{V_{f}\left( B_{p}\left( R\right) \right) }\int_{B_{p}\left(
2R_{\varepsilon }\right) }we^{-f} \\
\leq C\varepsilon +\frac{1}{V_{f}\left( B_{p}\left( R\right) \right) }%
\int_{B_{p}\left( 2R_{\varepsilon }\right) }we^{-f}\leq 2C\varepsilon .
\end{gather*}%
This proves (\ref{f2}).

The rest of the argument now follows verbatim \cite{LT}. For completeness we
sketch it below.

For $u,v\in \mathcal{H}^{1}\left( M\right) $ define 
\begin{equation*}
\left\langle \left\langle u,v\right\rangle \right\rangle
:=\lim_{R\rightarrow \infty }\frac{1}{V_{f}\left( B_{p}\left( R\right)
\right) }\int_{B_{p}\left( R\right) }\left\langle \nabla u,\nabla
v\right\rangle e^{-f}.
\end{equation*}%
This is well defined in view of (\ref{f2}). Also, $\left\langle \left\langle
,\right\rangle \right\rangle $ defines an inner product on 
\begin{equation*}
\mathcal{H}^{\prime }:=\left\{ u\in \mathcal{H}^{1}\left( M\right) :\
u\left( p\right) =0\right\} .
\end{equation*}%
Consider any finite dimensional subspace $\mathcal{H}^{\prime \prime }$ of $%
\mathcal{H}^{\prime }$, of dimension $l.$ Let $\left\{
u_{1},...,u_{l}\right\} $ be an orthonormal basis of $\left( \mathcal{H}%
^{\prime \prime },\left\langle \left\langle ,\right\rangle \right\rangle
\right) $ and define 
\begin{equation*}
F^{2}\left( x\right) :=\sum\limits_{i=1}^{l}u_{i}^{2}\left( x\right).
\end{equation*}%
Note $F$ is independent of the choice of $\left\{ u_{i}\right\} .$ For a
fixed point $x\in M,$ we may choose $\left\{ u_{i}\right\} $ so that $%
u_{i}\left( x\right) =0$ for all $i\neq 1$. Then it follows that%
\begin{eqnarray*}
F^{2}\left( x\right) &=&u_{1}^{2}\left( x\right) \ \ \text{and} \\
F\left( x\right) \nabla F\left( x\right) &=&u_{1}\left( x\right) \nabla
u_{1}\left( x\right) .
\end{eqnarray*}%
Since $\left\langle \left\langle u_{1},u_{1}\right\rangle \right\rangle =1,$
we have 
\begin{equation*}
\sup_{M}\left\vert \nabla u_{1}\right\vert =1.
\end{equation*}%
This shows that $\left\vert \nabla F\right\vert \left( x\right) \leq 1,$
too. Integrating along minimizing geodesics and using $F\left( p\right) =0,$
we get that $F\left( x\right) \leq r\left( x\right) .$ On the other hand, 
\begin{gather}
2\sum\limits_{i=1}^{l}\int_{B_{p}\left( R\right) }\left\vert \nabla
u_{i}\right\vert ^{2}e^{-f}=\int_{B_{p}\left( R\right) }\left( \Delta
_{f}F^{2}\right) e^{-f}  \label{f4} \\
\leq 2\int_{\partial B_{p}\left( R\right) }F\left\vert \nabla F\right\vert
e^{-f}\leq 2RA_{f}\left( \partial B_{p}\left( R\right) \right) .  \notag
\end{gather}%
Since $\left\{u_{i}\right\} $ is orthonormal with respect to $\left\langle
\left\langle ,\right\rangle \right\rangle ,$ for any $\varepsilon >0$ there
exists $R_{\varepsilon}$ such that for $R>R_{\varepsilon},$ 
\begin{equation*}
\sum\limits_{i=1}^{l}\frac{1}{V_{f}\left( B_{p}\left( R\right) \right) }%
\int_{B_{p}\left( R\right) }\left\vert \nabla u_{i}\right\vert
^{2}e^{-f}\geq l-\varepsilon.
\end{equation*}%
So, according to (\ref{f4}), for any $R\geq R_{\varepsilon },$ 
\begin{equation*}
\frac{l-\varepsilon }{R}\leq \frac{\left( V_{f}\left( B_{p}\left( R\right)
\right) \right) ^{\prime }}{V_{f}\left( B_{p}\left( R\right) \right) }.
\end{equation*}%
Integrating the inequality from $R_{\varepsilon }$ to $R,$ we then conclude
that 
\begin{equation}
\left( \frac{R}{R_{\varepsilon }}\right) ^{l-\varepsilon }\leq \frac{%
V_{f}\left( B_{p}\left( R\right) \right) }{V_{f}\left( B_{p}\left(
R_{\varepsilon }\right) \right) }  \label{f5}
\end{equation}%
On the other hand, according to Lemma \ref{Vol_Ric_Pos}, we have 
\begin{equation*}
V_{f}\left( B_{p}\left( R\right) \right) \leq CR^{n}.
\end{equation*}%
Plugging into (\ref{f5}), we conclude that 
\begin{equation*}
\dim \mathcal{H}^{1}\left( M\right) \leq n+1.
\end{equation*}
This proves the second claim of Theorem \ref{Polynomial}.

To prove that $\dim \mathcal{H}^{d}\left( M\right) \leq Cd^{\mu }$ for $d>1,$
we observe first that since $f$ is bounded we have for any $x\in M,$ 
\begin{equation*}
\frac{V\left( B_{x}\left( R\right) \right) }{V\left( B_{x}\left( r\right)
\right) }\leq C_{v}\left( \frac{R}{r}\right) ^{\mu }\ \ \text{for }r<R,
\end{equation*}%
where $C_{v}$ depends on $n$ and $\sup_{M}\left\vert f\right\vert .$
Moreover, the Sobolev inequality Lemma \ref{NS} and the Moser iteration
imply that we have a mean value inequality of the form 
\begin{equation*}
u^{2}\left( x\right) \leq C_{\mathcal{M}}\frac{1}{V\left( B_{x}\left(
R\right) \right) }\int_{B_{x}\left( R\right) }u^{2}
\end{equation*}%
for any nonnegative $f$-subharmonic function $u$ on $M.$ The result in \cite%
{L1} then implies $\dim \mathcal{H}^{d}\left( M\right) \leq Cd^{\mu }.$ The
theorem is proved.
\end{proof}

\section{Rigidity \label{Rigid}}

In this section we investigate the structure at infinity of smooth metric
measure spaces whose $\lambda _{1}\left( M\right) $ attains its upper bound
in Theorem \ref{Est_Ric_Pos}. We prove our result by using a Busemann
function argument, which is similar to the one in \cite{LW1}. A manifold $M$
is called connected at infinity if it has only one end.

\begin{theorem}
\label{Rigid_Ric_Pos} Let $\left( M,g,e^{-f}dv\right) $ be a smooth metric
measure space such that $Ric_{f}\geq 0$. Assume that $\lambda _{1}\left(
M\right) =\frac{1}{4}a^{2},$ where $a$ is the linear growth rate of $f.$
Then, either $M$ is connected at infinity or $M$ is isometric to $\mathbb{R}%
\times N$ for some compact manifold $N$.
\end{theorem}

Before proving the Theorem, we recall some terminology. First, a manifold is 
$f$-nonparabolic if $\Delta _{f}$ admits a positive Green's function.
Otherwise, it is called $f$-parabolic. For an end of the manifold, the same
definition applies, where the Green's function now refers to the one
satisfying the Neumann boundary conditions. We will divide our proof into
two cases according to the type of ends of $M.$ In fact, the result for $f$%
-nonparabolic ends does not require any assumption on $f$ or $\lambda
_{1}\left( M\right).$ We state it in the following Lemma.

\begin{lemma}
\label{nonparabolic} Let $\left( M,g,e^{-f}dv\right) $ be a smooth metric
measure space with $Ric_{f}\geq 0$. Then $M$ has at most one $f$%
-nonparabolic end.
\end{lemma}

\begin{proof}[Proof of Lemma \protect\ref{nonparabolic}]
Suppose $M$ has two $f$-nonparabolic ends. Then $M$ admits a positive
non-constant bounded $f$-harmonic function $v$ with $\int_{M}\left\vert
\nabla v\right\vert ^{2}e^{-f}<\infty .$ This kind of result was first
discovered by Li and Tam (see \cite{L}). Now according to a result of
Brighton \cite{B}, since $v$ is bounded, it must be a constant function.
This is a contradiction.

Alternatively, the Lemma may be proved as follows. Using the Bochner formula 
\begin{equation*}
\frac{1}{2}\Delta _{f}\left\vert \nabla v\right\vert ^{2}=\left\vert
v_{ij}\right\vert ^{2}+\left\langle \nabla \Delta _{f}v,\nabla
v\right\rangle +Ric_{f}\left( \nabla v,\nabla v\right) \geq \left\vert
v_{ij}\right\vert ^{2}
\end{equation*}%
and a cut-off argument, we get 
\begin{gather*}
2\int_{M}\left\vert v_{ij}\right\vert ^{2}e^{-f}\phi ^{2}\le \int_{M}\left(
\Delta _{f}\left\vert \nabla v\right\vert ^{2}\right) \phi ^{2}e^{-f} \\
=-\int_{M}\left\langle \nabla \left\vert \nabla v\right\vert ^{2},\nabla
\phi ^{2}\right\rangle e^{-f} \\
\leq \int_{M}\left\vert v_{ij}\right\vert ^{2}\phi
^{2}e^{-f}+4\int_{M}\left\vert \nabla v\right\vert ^{2}\left\vert \nabla
\phi \right\vert ^{2}e^{-f}.
\end{gather*}%
Thus, 
\begin{equation*}
\int_{M}\left\vert v_{ij}\right\vert ^{2}e^{-f}\phi ^{2}\le
4\int_{M}\left\vert \nabla v\right\vert ^{2}\left\vert \nabla \phi
\right\vert ^{2}e^{-f}.
\end{equation*}%
Since $\int_{M}\left\vert \nabla v\right\vert ^{2}e^{-f}<\infty ,$ the right
hand side goes to $0$ after taking the limit $R\to \infty$ by choosing $\phi$
to be the standard cut-off function so that $\phi=1$ on $B_p(R)$ and $\phi=0$
outside $B_p(2R).$ This forces $v_{ij}=0.$ Therefore, $|\nabla v|$ must be a
constant on $M$. Now the weighted volume $V_f(M)=\int_M e^{-f}dv=\infty$ as $%
M$ is $f$-nonparabolic. Using $\int_{M}\left\vert \nabla v\right\vert
^{2}e^{-f}<\infty$ again, we conclude $|\nabla v|=0$ and $v$ is a constant.
This finishes our proof.
\end{proof}

We now prove Theorem \ref{Rigid_Ric_Pos}.

\begin{proof}[Proof of Theorem \protect\ref{Rigid_Ric_Pos} ]
First, since $\lambda _{1}\left( M\right) >0,$ we know that $M$ is $f$%
-nonparabolic. This follows exactly as in the case $f=$constant. We refer to 
\cite{L} for details.

Let us assume that $M$ has at least two ends. By Lemma \ref{nonparabolic}
exactly one end is $f-$nonparabolic, and all other are $f-$parabolic. We
denote $E$ to be the $f$-nonparabolic end and let $F:=M\backslash E$. Then $%
F $ is an $f$-parabolic end.

First, we claim the fact that $\lambda _{1}\left( M\right) =\frac{1}{4}a^{2}$
implies 
\begin{equation}
V_{f}\left( F\backslash B_{p}\left( R\right) \right) \leq Ce^{-aR}.
\label{eq1}
\end{equation}%
This claim can be verified by following the argument of Li and Wang in \cite%
{LW}. Since they only use the variational principle for the bottom spectrum
and integration by parts, it is easy to check that their estimates can be
carried over to our setting. We shall omit the details here.

For a geodesic ray $\gamma $ contained in the end $F,$ define the associated
Busemann function by 
\begin{equation*}
\beta \left( x\right) :=\lim_{t\rightarrow \infty }\left( t-d\left( x,\gamma
\left( t\right) \right) \right) .
\end{equation*}%
Denote by $\tau _{t}\left( s\right) $ the normal minimizing geodesic from $%
\gamma \left( t\right) $ to $x$. According to (\ref{v2}) we have%
\begin{equation*}
\Delta _{f}\left( d\left( \gamma \left( t\right) ,x\right) \right) \leq 
\frac{n-1}{r}-\frac{2}{r}f\left( x\right) +\frac{2}{r^{2}}%
\int_{0}^{r}f\left( s\right) ds,
\end{equation*}%
where $r:=d\left( \gamma \left( t\right) ,x\right) $ and $f\left( s\right)
:=f\left( \tau _{t}\left( s\right) \right) .$ Let us assume that there exist 
$\alpha >0$ and $C>0$ such that $\left\vert f\left( z\right) \right\vert
\leq \alpha d\left( p,z\right) +C,$ for all $z\in M.$ Then, 
\begin{eqnarray*}
\left\vert f\left( s\right) \right\vert  &=&\left\vert f\left( \tau
_{t}\left( s\right) \right) \right\vert \leq \alpha d\left( p,\tau
_{t}\left( s\right) \right) +C \\
&\leq &\alpha \left( d\left( p,x\right) +d\left( x,\tau _{t}\left( s\right)
\right) \right) +C \\
&=&\alpha \left( r-s\right) +\left( \alpha d\left( p,x\right) +C\right) ,
\end{eqnarray*}%
for any $0<s<r.$ Therefore, since $\left\vert f\left( x\right) \right\vert
\leq \alpha d\left( p,x\right) +C,$ we conclude 
\begin{eqnarray*}
\Delta _{f}\left( d\left( \gamma \left( t\right) ,x\right) \right)  &\leq &%
\frac{n-1+4\left( \alpha d\left( p,x\right) +C\right) }{r}+\frac{2}{r^{2}}%
\int_{0}^{r}\alpha (r-s)ds \\
&=&\frac{n-1+4\left( \alpha d\left( p,x\right) +C\right) }{r}+\alpha .
\end{eqnarray*}%
Now, by using this and the definition of the Busemann function, it is
standard to see that the following inequality holds in the sense of
distributions: 
\begin{equation*}
\Delta _{f}\beta (x) \geq -\alpha .
\end{equation*}%
Now, if $a$ denotes the linear growth rate of $f$ i.e., it is the infimum of
such $\alpha ,$ it follows that 
\begin{equation}
\Delta _{f}\beta \left( x\right) \geq -a.  \label{eq2}
\end{equation}

To estimate the volume growth of the f-nonparabolic end $E$ we will use (\ref%
{eq2}) and the fact that $\left\vert \nabla \beta \right\vert =1.$ Note 
\begin{equation*}
\Delta _{f}e^{a\beta }\geq 0.
\end{equation*}%
We also know that on the end $E$ the Busemann function is equivalent to the
distance function, that is, there exists a constant $C$ such that 
\begin{equation*}
-r\left( x\right) -C\leq \beta \left( x\right) \leq -r\left( x\right) +C
\end{equation*}%
for $x\in E,$ see \cite{LW1}. Integrating $\Delta _{f}e^{a\beta }$ on $%
\left\{ -t\leq \beta \leq -r\right\} \cap E$ and using Stokes theorem we get%
\begin{eqnarray*}
0 &\leq &\frac{1}{a}\int_{\left\{ -t\leq \beta \leq -r\right\} \cap E}\left(
\Delta _{f}e^{a\beta }\right) e^{-f} \\
&=&\int_{\left\{ \beta =-r\right\} \cap E}e^{a\beta }e^{-f}-\int_{\left\{
\beta =-t\right\} \cap E}e^{a\beta }e^{-f} \\
&=&e^{-ar}A_{f}\left( \left\{ \beta =-r\right\} \cap E\right)
-e^{-at}A_{f}\left( \left\{ \beta =-t\right\} \cap E\right) ,
\end{eqnarray*}%
where we have used $\left\vert \nabla \beta \right\vert =1$. Here, $%
A_{f}\left( \Omega \right) $ denotes the area of the set $\Omega $ with
respect to the weighted area form. This shows for a fixed $r$ and all $t>r,$ 
\begin{equation*}
A_{f}\left( \left\{ \beta =-t\right\} \cap E\right) \leq C\left( r\right)
e^{at}.
\end{equation*}%
Integrating with respect to $t$ from $r$ to $R$ we obtain an upper bound for
the volume of the sublevel sets of the Busemann function. Since the Busemann
function is equivalent with the distance function on $E,$ it follows that
for $R>0,$ 
\begin{equation}
V_{f}\left( B_{p}\left( R\right) \cap E\right) \leq Ce^{aR}.  \label{eq3}
\end{equation}

Consider the function 
\begin{equation*}
B:=e^{\frac{1}{2}a\beta }.
\end{equation*}%
By (\ref{eq2}) and $\left\vert \nabla \beta \right\vert =1,$ we have%
\begin{equation}
\Delta _{f}B\geq -\frac{1}{4}a^{2}B.  \label{eq4}
\end{equation}%
We now use $B$ as a test function for the variational formula of $\lambda
_{1}\left( M\right) =\frac{1}{4}a^{2}.$ For this sake, we define a cut-off
function $\phi $ with support in $B_{p}\left( 2R\right) $ such that $\phi =1$
on $B_{p}\left( R\right) $ and $\left\vert \nabla \phi \right\vert \leq 
\frac{C}{R}.$ Then, according to the variational characterization of $%
\lambda_1(M),$ we have 
\begin{gather*}
\frac{1}{4}a^{2}\int_{M}\left( B\phi \right) ^{2}e^{-f}\leq
\int_{M}\left\vert \nabla \left( B\phi \right) \right\vert ^{2}e^{-f} \\
=\int_{M}\left\vert \nabla B\right\vert ^{2}\phi ^{2}e^{-f}+\frac{1}{2}%
\int_{M}\left\langle \nabla B^{2},\nabla \phi ^{2}\right\rangle
e^{-f}+\int_{M}\left\vert \nabla \phi \right\vert ^{2}B^{2}e^{-f} \\
=\int_{M}\left\vert \nabla B\right\vert ^{2}\phi ^{2}e^{-f}-\frac{1}{2}%
\int_{M}\left( \Delta _{f}B^{2}\right) \phi ^{2}e^{-f}+\int_{M}\left\vert
\nabla \phi \right\vert ^{2}B^{2}e^{-f} \\
=-\int_{M}B\left( \Delta _{f}B\right) \phi ^{2}e^{-f}+\int_{M}\left\vert
\nabla \phi \right\vert ^{2}B^{2}e^{-f} \\
=\frac{1}{4}a^{2}\int_{M}B^{2}\phi ^{2}e^{-f}+\int_{M}\left\vert \nabla \phi
\right\vert ^{2}B^{2}e^{-f} -\int_{M}B\left( \Delta _{f}B +\frac{1}{4}a^{2}
B\right) \phi ^{2}e^{-f}.
\end{gather*}%
Therefore, 
\begin{equation*}
\int_{M}B\left( \Delta _{f}B +\frac{1}{4}a^{2} B\right) \phi ^{2}e^{-f} \le
\int_{M}\left\vert \nabla \phi \right\vert ^{2}B^{2}e^{-f}.
\end{equation*}%
From (\ref{eq1}) and (\ref{eq3}) and the construction of $B$ we have 
\begin{equation*}
\int_{M}\left\vert \nabla \phi \right\vert ^{2}B^{2}e^{-f}\leq \frac{C}{R}%
\rightarrow 0.
\end{equation*}
In view of (\ref{eq4}), we conclude $\Delta _{f}B+\frac{1}{4}a^{2}B=0.$ In
particular, $B$ is smooth. Therefore, 
\begin{equation*}
\Delta _{f}\beta =-a\ \ \text{and\ }\ \left\vert \nabla \beta \right\vert =1
\end{equation*}%
hold everywhere on $M.$ Now, by the Bochner formula, 
\begin{equation*}
0=\frac{1}{2}\Delta _{f}\left\vert \nabla \beta \right\vert ^{2}=\left\vert
\beta _{ij}\right\vert ^{2}+\left\langle \nabla \Delta _{f}\beta ,\nabla
\beta \right\rangle +Ric_{f}\left( \nabla \beta ,\nabla \beta \right) \geq
\left\vert \beta _{ij}\right\vert ^{2}.
\end{equation*}%
This proves $\beta _{ij}=0,$ which implies $M$ is a direct product given by $%
M=\mathbb{R}\times N.$ That $N$ is compact follows from the assumption of $M$
having two ends. This proves the Theorem.
\end{proof}

We now discuss an important application of this result to the gradient Ricci
solitons. Recall that a gradient steady Ricci soliton is a manifold $\left(
M,g,f\right) $ such that $R_{ij}+f_{ij}=0.$ Here we prove that gradient
steady Ricci solitons are either connected at infinity or they are trivial,
that is, isometric to a Ricci flat cylinder.

\begin{theorem}
\label{S} Let $\left( M,g,f\right) $ be a gradient steady Ricci soliton.
Then either $M$ is connected at infinity or $M$ is isometric to $\mathbb{R}%
\times N$ for a compact Ricci flat manifold $N.$
\end{theorem}

\begin{proof}[Proof of Theorem \protect\ref{S} ]
Let us assume that $M$ has at least two ends. It is known that for any
gradient steady Ricci soliton there exists $a>0$ so that $\left\vert \nabla
f\right\vert ^{2}+S=a^{2},$ see (\ref{v6}). By Proposition \ref{Spec_Steady}
we know that $\lambda _{1}\left( M\right) =\frac{1}{4}a^{2}.$ Then, by
Theorem \ref{Rigid_Ric_Pos} we conclude $M=\mathbb{R}\times N$ for some
compact manifold $N.$ It is easy to see $N$ has to be a steady Ricci soliton
also. However, any compact steady gradient Ricci soliton must be Ricci flat.
This concludes the proof.
\end{proof}

{\small DEPARTMENT OF MATHEMATICS, COLUMBIA UNIVERSITY }

{\small NEW YORK, NY 10027}\newline
{\small E-mail address: omuntean@math.columbia.edu}

\bigskip

{\small SCHOOL OF MATHEMATICS, UNIVERSITY OF MINNESOTA }

{\small MINNEAPOLIS, MN 55455} \newline
{\small E-mail address: jiaping@math.umn.edu}

\end{document}